\documentclass[11pt]{amsart}
\usepackage{amscd,graphicx,accents,url,xcolor}
\usepackage{atbegshi}
\AtBeginShipoutFirst{\vspace{-.25in}\small\color{red!50!black}\centerline{To appear in: \emph{Analysis and Numerics of Partial Differential Equations},}
\centerline{U.~Gianazza, F.~Brezzi, P.~Colli~Franzone, and G.~Gilardi, eds., Springer 2013.}}

\calclayout
\numberwithin{equation}{section}

\newcommand{\R}{\mathbb{R}}
\renewcommand\H{{\mathcal H}}
\newcommand\J{{\mathcal J}}
\renewcommand\P{{\mathcal P}}
\newcommand\Q{{\mathcal Q}}
\renewcommand\S{{\mathcal S}}
\newcommand\T{{\mathcal T}}
\newcommand\x{\times}
\newcommand\<{\langle}
\renewcommand\>{\rangle}
\newcommand\ldeg{\operatorname{ldeg}}
\newcommand\sdeg{\operatorname{sdeg}}
\newtheorem{thm}{Theorem}[section] 
\newtheorem{lemma}[thm]{Lemma}

\newtheorem{cor}[thm]{Corollary}

\DeclareMathOperator{\sign}{sign}

\DeclareMathOperator{\grad}{grad}
\DeclareMathOperator{\curl}{curl}

\DeclareMathOperator{\tr}{tr}
\let\div\undefined
\DeclareMathOperator{\div}{div}

\newcommand\Alt{\operatorname{Alt}}

\newcommand\skw{\operatorname{skw}}
\newcommand\vol{\mathsf{vol}}
\newcommand\0{\ring}
\newcommand\Null{\operatorname{\mathcal N}}

\begin{document}
\title{Spaces of finite element differential forms}

\dedicatory{In memory of Enrico Magenes, in gratitude for his
deep and elegant mathematics,\\ which taught us, and his profound
humanity, which inspired us.}

\thanks{The work of the first author was supported by NSF grant DMS-1115291.}

\author{Douglas N. Arnold}

\address{School of Mathematics, University of Minnesota, Minneapolis, Minnesota 55455}

\email{arnold@umn.edu}
\urladdr{http://www.ima.umn.edu/\~{}arnold}


%\ccode{AMS Subject Classification: 65N30, 74S05}

\begin{abstract}
We discuss the construction of finite element spaces of differential forms
which satisfy the crucial assumptions of the finite element exterior calculus,
namely that they can be assembled into subcomplexes of the de~Rham complex
which admit commuting projections.  We present two families
of spaces in the case of simplicial meshes, and two other families
in the case of cubical meshes.  We make use of the exterior calculus and
the Koszul complex to define and understand the spaces.  These tools allow
us to treat a wide variety of situations, which are often treated separately,
in a unified fashion.
\end{abstract}

\maketitle

\section{Introduction}
The gradient, curl, and divergence are the most fundamental operators of
vector calculus, appearing throughout the differential equations of mathematical physics
and other applications.  The finite element solution of such equations
requires finite element subspaces of the
natural Hilbert space domains of these operators, namely $H^1$, $H(\curl)$, and $H(\div)$.
The construction of subspaces with desirable properties has been
an active research topic for half a century.  Exterior calculus provides a framework in
which these fundamental operators and spaces are unified and generalized, and their properties
and inter-relations clarified.  Each of the operators is viewed as a particular case
of the exterior derivative operator $d=d^k$ taking differential $k$-forms on
some domain $\Omega\subset\R^n$ to differential
$(k+1)$-forms.  We regard $d^k$ as an unbounded operator between the
Hilbert spaces $L^2\Lambda^k$ and $L^2\Lambda^{k+1}$ consisting of
differential forms with $L^2$ coefficients.
The domain of $d^k$ is the Hilbert space
\begin{equation}\label{hlambda}
H\Lambda^k = \{\, u\in L^2\Lambda^k\,|\, du\in L^2\Lambda^{k+1}\,\},
\end{equation}
and all the $d^k$ and their domains combine to form the $L^2$ de~Rham complex
$$
0 \to H\Lambda^0 \xrightarrow{d^0} H\Lambda^1 \xrightarrow{d^1} \cdots
\xrightarrow{d^{n-1}} H\Lambda^n \to 0.
$$
Differential $0$-forms and $n$-forms may be identified simply with functions
on $\Omega$ and differential $1$-forms and $(n-1)$-forms may be identified
with vector fields.  In three dimensions, we may use these proxies to write the de~Rham complex
as
$$
0 \to H^1 \xrightarrow{\grad} H(\curl) \xrightarrow{\curl} H(\div) \xrightarrow{\div} L^2 \to 0.
$$

The finite element exterior calculus (FEEC) is a theory developed in the last
decade \cite{ArnoldICM,acta,bulletin} which enables the development and analysis
of finite element spaces of differential forms.  One major part of FEEC is carried
out in the framework of Hilbert complexes, of which the $L^2$ de~Rham complex is
the most canonical example.  One important outcome of FEEC is the realization
that the finite dimensional subspaces
$\Lambda^k_h\subset H\Lambda^k$ used in Galerkin discretizations
of a variety of differential equations
involving differential $k$-forms should satisfy two basic assumptions, beyond
the obvious requirement that the spaces have good approximation properties.  The first
assumption is that the subspaces form a \emph{subcomplex} of the de~Rham complex,
i.e., that $d \Lambda^k_h\subset \Lambda^{k+1}_h$.  The second is that there exist
projection operators $\pi^k_h$ from $H\Lambda^k$ to $\Lambda^k_h$ which commute with $d$
in the sense that the following diagram commutes:
$$
\begin{CD}
H\Lambda^0 @>d>> H\Lambda^1 @>d>> \cdots @>d>> H\Lambda^{n-1\kern-1mm} @>d>> H\Lambda^n
\\
@V\pi^0_hVV  @V\pi^1_hVV  @. @V\pi^{n-1}_hVV  @V\pi^n_hVV
\\
\Lambda^0_h @>d>> \Lambda^1_h  @>d>> \cdots @>d>>  \Lambda^{n-1}_h  @>d>> \Lambda^n_h
\end{CD}
$$

The second major part of FEEC, into which the present exposition falls, is concerned
with the construction of specific finite element spaces $\Lambda^k_h$ of differential forms.
A special role is played by two families of finite element spaces $\P_r^-\Lambda^k(\T_h)$
and $\P_r\Lambda^k(\T_h)$, defined for any dimension $n$, any simplicial mesh $\T_h$,
any polynomial degree $r\ge1$,
and any form degree $0\le k\le n$.  Both these spaces are subspaces of $H\Lambda^k(\Omega)$.
The $\P_r^-\Lambda^k$ spaces with increasing $k$ and constant $r$ form a subcomplex
of $L^2$ de~Rham complex which admits commuting projections.  The same is true of
the $\P_r\Lambda^k$ family, except in that case the polynomial degree $r$ decreases
as the form degree $k$ increases.

We also discuss cubical meshes.  In this case, there is a well-known family of elements,
denoted by $\Q_r^-\Lambda^k$ in our notation, obtained by a tensor product construction.
As for the $\P_r^-\Lambda^k$ family, the $\Q_r^-\Lambda^k$ spaces with constant degree $r$
combine to form a de~Rham subcomplex with
commuting projections. 
We also discuss a recently discovered second family on cubical meshes, the $\S_r\Lambda^k$
family of \cite{cubicderham}.  Like the $\P_r\Lambda^k$ family, the de~Rham subcomplexes for this family
are obtained with decreasing degree.  Moreover for large $r$, the $\dim\S_r\Lambda^k(\T_h)$
is much smaller dimension than $\dim\Q_r^-\Lambda^k$. 
The finite element subspaces of $H^1$, $H(\curl)$, and $H(\div)$ from this family in three dimensions
are new.

The remainder of the paper is organized as follows.  In the next section we cover
some preliminary material (which the more expert reader may wish to skip).
We recall the construction of finite element spaces from spaces of shape functions
and unisolvent degrees of freedom.  To illustrate we discuss the Lagrange elements
and carry out the proof of unisolvence in a manner that will guide our treatment
of differential form spaces of higher degree.  We also give a brief summary of
those aspects of exterior calculus most relevant to us.  In section~\ref{sec:simp} we discuss the two
primary families of finite element spaces for differential forms on simplicial meshes
mentioned above.  A key role is played by the Koszul complex, which is introduced
in this section.  Then, in Theorem~\ref{unisolv}, we give a proof of unisolvence for the $\P_r^-$ family
which we believe to be simpler than has appeared heretofore
(a similar proof could be given for the $\P_r$ family as well).
In the final section we review the two families mentioned for cubical meshes, including
a description, without proofs, of the recently discovered $\S_r$ family.

\section{Preliminaries}
\subsection{The assembly of finite element spaces}
Recalling the definition of a finite element space \cite{ciarlet}, we assume that
the domain $\Omega\subset\R^n$ is triangulated by finite elements,
i.e., its closure is the union of a finite set $\T_h$ of closed convex polyhedral
elements with nonempty interiors such that the intersection of any two elements
is either empty or is a common face of each of some dimension.  We denote
by $\Delta_d(T)$ the set of faces of $T$ of dimensions $d$, so, for example,
$\Delta_0(T)$ is the set of vertices of $T$, and $\Delta_n(T)$ is the singleton
set whose only element is $T$.  We also define $\Delta(T)=\bigcup_{0\le d\le n}\Delta_d(T)$,
the set of all faces of $T$.
In this paper we consider the two cases
of simplicial elements, in which each element $T$ of the triangulation is an $n$-simplex,
and cubical elements, in which element is an $n$-box (i.e., the Cartesian product
of $n$ intervals).  To define a finite element space $\Lambda^k_h\subset H\Lambda^k(\Omega)$,
we must supply, for each element $T\in\T_h$,
\begin{enumerate}
 \item A finite dimensional space $V(T)$, called the space of \emph{shape
 functions}, consisting of differential $k$-forms on $T$ with polynomial coefficients.
 The finite element space will consist of functions $u$ which belong to the shape function spaces piecewise
 in the sense that $u|_T\in V(T)$ for all $T\in\T_h$
 (allowing the possibility that $u$ is multiply-valued
 on faces of dimension $<n$).
 \item A set of functionals $V(T)\to\R$, called the \emph{degrees of freedom}, which are \emph{unisolvent}
 (i.e., which form a basis for the dual space $V(T)^*$) and such that
 each degree of freedom is associated to a specific face of $f\in\Delta(T)$.
\end{enumerate}
It is assumed that when two distinct elements $T_1$ and $T_2$ intersect in a common face
$f$, the degrees of
freedom of $T_1$ and $T_2$ which are associated to $f$ are in a specific
1-to-1 correspondence.  If $u$ is a function which belongs to the shape function spaces
piecewise, then we say that the degrees of freedom are single-valued on $u$ if
whenever two elements $T_1\ne T_2$ meet in a common face, then the corresponding
degrees of freedom associated to the face take the same value on $u|_{T_1}$ and $u|_{T_2}$, respectively.
With these ingredients, the finite element space $\Lambda^k_h$ associated to
the choice of triangulation $\T_h$, the shape function spaces $V(T)$, and the degrees of
freedom, is defined as the set of all $k$-forms on $\Omega$ which belong to the shape
function spaces piecewise and for which all the degrees of freedom are single-valued.

The choice of the degrees of freedom associated to faces of dimension
$d<n$ determine the interelement continuity imposed on the finite element subspace.
The use of degrees of freedom to specify the continuity, rather than imposing
the continuity a priori in the definition of the finite element space, is of great practical significance in that
it assures that the finite element space can be implemented efficiently.  The
dimension of the space is known (it is just the sum over the faces of the triangulation of the number of
degrees of freedom associated to the face) and it depends only on the topology
of the triangulation, not on the coordinates
of the element vertices.  Moreover, the degrees of freedom lead to a computable
basis for $\Lambda^k_h$ in which each basis element is associated to one degree of freedom.
Further, the basis is local, in that the basis element for a degree of freedom associated
to a face $f$ is nonzero only on the elements that contain $f$.

The finite element space so defined does not depend on the specific choice
of degrees of freedom in $V(T)^*$, but only on the span of the degrees of freedom
associated to each face $f$ of $T$, and we shall generally specify only
the span, rather than a specific choice of basis for it.

\subsection{The Lagrange finite element family}
To illustrate these definitions and motivate the constructions
for differential forms, we consider the simplest
example, the Lagrange family of finite element subspaces of $H^1=H\Lambda^0$.  The Lagrange
space, which we denote $\P_r\Lambda^0(\T_h)$ in anticipation of its generalization below,
is defined for any simplicial triangulation $\T_h$ in $\R^n$ and any polynomial degree $r\ge 1$.
The shape function space is $V(T)=\P_r(T)$, the space of all polynomial functions on $T$ of
degree at most $r$.  For a face $f$ of $T$ of dimension $d$, the span of the associated
degrees of freedom are the functionals
\begin{equation}\label{lagdofs}
u\in\P_r(T) \mapsto \int_f (\tr_f u) q, \ q\in \P_{r-d-1}(f), \ f\in\Delta(T).
\end{equation}
In interpreting this, we understand the space $\P_s(f)$ to be the space $\R$ of constants
if $f$ is $0$-dimensional (a single vertex) and $s\ge 0$.  Also the space
$\P_s(f)=0$ if $s<0$ and $f$ is arbitrary.  The notation $\tr_f u$ denotes the trace
of $u$ on $f$, i.e., its restriction.  Thus there is one degree of freedom
associated to each vertex $v$, namely the evaluation functional $u\mapsto u(v)$.
For $r\ge2$ there are also degrees of freedom associated to the edges $e$ of $T$, namely
the moments of $u$ on the edge of degree at most $r-2$:
$$
u\mapsto  \int_e (\tr_e u)q, \ q\in \P_{r-2}(e).
$$
For $r\ge 3$ there are degrees of freedom associated to the $2$-faces, namely moments of degree
at most $r-3$, etc.  This is often indicated in a degree of freedom diagram, like that
of Figure~\ref{fg:dofs}, in which the number of symbols drawn in the interior of a face is equal
to the number of degrees of freedom associated to the face.
\begin{figure}[htb]
 \includegraphics[height=1.5in]{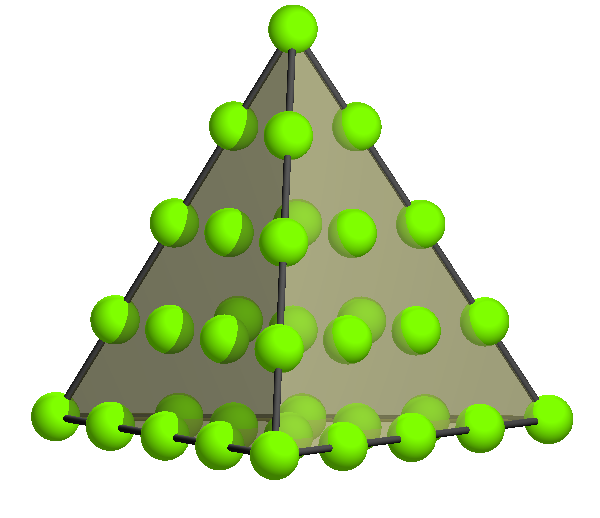}
\caption[]{Degrees of freedom for the Lagrange quartic space $\P_4\Lambda^0$ in $3$ dimensions.}\label{fg:dofs}
\end{figure}

A requirement of the definition of a finite element space is that the degrees of freedom be unisolvent.
We present the proof for Lagrange elements in detail, since it will guide us when it comes to verifying unisolvence for more
complicated spaces.
\begin{thm}[Unisolvence for the Lagrange elements]
For any $r\ge1$ and any $n$-simplex $T$, the degrees of freedom \eqref{lagdofs} are unisolvent on $V(T)=\P_r(T)$.
\end{thm}
\begin{proof}
It suffices to verify, first, that the number of degrees of freedom proposed for $T$ does not exceed $\dim V(T)$,
and, second, that if all the degrees of freedom vanish when applied to some $u\in V(T)$, then $u\equiv 0$.
For the first claim, we have by \eqref{lagdofs} that
the total number of degrees of freedom is at most
$$
\sum_{d=0}^n\#\Delta_d(T)\dim\P_{r-d-1}(\R^d)=
\sum_{d=0}^n\binom{n+1}{d+1}\binom{r-1}{d}=\binom{n+r}{n}=\dim\P_r(T),
$$
where the second equality is a binomial identity which comes from expanding in the equation
$(1+x)^{n+1}(1+x)^{r-1}=(1+x)^{n+r}$ and comparing the coefficients of $x^n$ on both sides.

We prove the second claim by induction on the dimension $n$, the case $n=0$ being trivial.
Suppose that $u\in\P_r(T)$ for some simplex $T$ of dimension $n$
and that all the degrees of freedom in \eqref{lagdofs} vanish.  We wish to show that $u$ vanishes.
Let $F\in\Delta_{n-1}(T)$ be a facet of $T$, 
and consider $\tr_{F} u$, which is a polynomial function of at most degree $r$ on the $(n-1)$-dimensional
simplex $F$, i.e., it belongs
to $\P_r(F)$.  Moreover, if we replace $T$ by $F$ and $u$ by $\tr_{F}u$ in \eqref{lagdofs}, the
resulting functionals vanish by assumption (using
the obvious fact that $\tr_f\tr_{F} u = \tr_f u$ for $f\subset F\subset T$).  By induction we
conclude that $\tr_{F} u$ vanishes on all the facets $F$ of $T$.  Therefore, $u$ is divisible by the
barycentric coordinate function $\lambda_i$ which vanishes on $F$, and, since this holds
for all facets, $u=(\prod_{i=0}^n\lambda_i)p$
for some $p\in\P_{r-n-1}(T)$.  Taking $f=T$ and $q=p$ in \eqref{lagdofs} we conclude that
$$
\int_T (\prod_{i=0}^n\lambda_i)p^2 =0,
$$ which implies that $p$ vanishes on $T$, and so $u$ does as well.
\end{proof}

Let us note some features of the proof, which will be common to the unisolvence proofs for all of
the finite element spaces we discuss here.  After a dimension count to verify that the proposed
degrees of freedom are correct in number, or at least no more than required, the proof
proceeded by induction on the number of
space dimensions.  The inductive step relied on a trace property of the shape function space $V(T)=\P_r(T)$
for the family, namely that $\tr_F V(T)\subset V(F)$.  Moreover, it used a similar trace
property for the degrees of freedom: if $\xi_F\in V(F)^*$ is a degree of freedom for $V(F)$, 
then the pullback $\xi_F\circ \tr_F\in V(T)^*$ is a degree of freedom for
$V(T)$.  The induction reduced the unisolvence proof to verifying that if $u\in \0V(T)$, the space of functions in $V(T)$ whose
trace vanishes on the entire boundary, and if the interior degrees of freedom (those associated to $T$ itself) of $u$
vanish, then $u$ itself vanishes, which we showed by explicit construction.

Finally, we note that the continuity implied by the degrees of freedom is exactly what
is required to insure that the Lagrange finite element space is contained in $H^1$:
\begin{equation}\label{lagcont}
 \P_r\Lambda^0(\T_h) = \{\,u\in H^1(\Omega)\,|\, u\text{ belongs to $\P_r(T)$ piecewise}\,\}.
\end{equation}
Indeed, a piecewise smooth function belongs to $H^1(\Omega)$ if and only if its traces
on faces are single-valued.  Thus if a function in $H^1(\Omega)$ belongs piecewise to $\P_r(T)$,
its traces
are single-valued, so the degrees of freedom
are single-valued, and the function belongs to $\P_r\Lambda^0(\T_h)$.
On the other hand, if the function belongs to $\P_r\Lambda^0(\T_h)$,
its traces on faces are single-valued, since, as we saw in the course of the
unisolvence proof, they are determined by the degrees of freedom.  Thus
the function belongs to $H^1(\Omega)$.

\subsection{Exterior calculus}
For the convenience of readers less familiar with differential forms
and exterior calculus we now briefly review
key definitions and properties.  We begin with the space of \emph{algebraic} $k$-forms on
$V$:
$\Alt^k V  = \{\, L:V^k\to\R\,|\, \text{$k$-linear, skew-symmetric}\,\}$, where the
multilinear form $L$ is skew-symmetric, or alternating, if it changes sign under
the interchange of any two of its arguments.
The skew-symmetry condition is vacuous if $k<2$, so
$\Alt^1V=V^*$ and, by convention, $\Alt^0V=\R$.  If $\omega$ is any $k$-linear map  $V^k\to \R$,
then $\skw\omega\in\Alt^kV$ where
$$
(\skw \omega)(v_1,\ldots,v_k) = \frac{1}{k!}\sum_{\sigma}\sign(\sigma)\omega(v_{\sigma_1},\ldots,v_{\sigma_k}),
$$
with the sum taken over all the permutations of the integers $1$ to $k$.
The wedge product $\Alt^kV \x \Alt^lV \to \Alt^{k+l}V$ is defined
$$
\omega\wedge\mu = \binom{k+l}{k} \skw(\omega\otimes\mu), \quad \omega\in\Alt^kV,\ \mu\in\Alt^lV.
$$
Let $v_1,\ldots,v_n$ form a basis for $V$.  Denoting by
$$
\Sigma(k,n) = \{\,(\sigma_1,\ldots,\sigma_k)\in \mathbb N^k\,|\,1\le \sigma_1<\cdots<\sigma_k\le n\,\},
$$
an element of $\Alt^kV$ is completely determined by the values it assigns to the $k$-tuples $(v_{\sigma_1},\ldots, v_{\sigma_k})$,
$\sigma\in\Sigma_k$.
Moreover, these values can be assigned arbitrarily.  In fact, the $k$-form
$\mu_{\sigma_1}\wedge\cdots\wedge\mu_{\sigma_k}$,
where $\mu_1,\ldots,\mu_n$ is the dual basis to $v_1,\ldots,v_n$, takes the $k$-tuple
$(v_{\sigma_1},\ldots, v_{\sigma_k})$ to $1$, and the other such $k$-tuples to $0$.
Thus $\dim\Alt^k V=\binom{n}{k}$, where $n=\dim V$.

We define differential forms on an arbitrary manifold, since we will be using
them both when the manifold is a domain in $\R^n$ and when it is the boundary of such a domain.
A differential $k$-form on a manifold $\Omega$ is a map $\omega$ which takes each point $x\in\Omega$
to an element $\omega_x\in\Alt^k T_x\Omega$, where $T_x\Omega$ is the tangent space to $\Omega$ at $x$.
In other language, $\omega$ is a skew-symmetric covariant tensor field
on $\Omega$ of order $k$.  In particular, a differential $0$-form is just a real-valued
function on $\Omega$ and a differential $1$-form is a covector field.
In the case $\Omega$ is a domain in $\R^n$, then
each tangent space can be identified with $\R^n$, and a differential $k$-form is simply a map
$\Omega\to\Alt^k\R^n$.  In this context, it is common to denote the dual basis to the canonical
basis for $\R^n$ by $dx^1,\ldots,dx^n$, so $dx^k$ applied to a vector $v=(v^1,\ldots,v^n)\in\R^n$
is its $k$th component $v^k$.  With this notation,  an arbitrary differential $k$-form can be written
$$
u(x) = \sum_{\sigma\in\Sigma(k,n)} a_\sigma(x)\,dx^{\sigma_1}\wedge\cdots\wedge dx^{\sigma_k},
$$
for some coefficients $a_\sigma:\Omega\to\R$.

Three basic operations on differential forms are the exterior derivative, the form integral, and
the pullback.  The exterior derivative $d\omega$ of a $k$-form $\omega$ is a $(k+1)$-form.  In
the case of a domain in $\R^n$, it is given by the intuitive formula
$$
d(a_\sigma\,dx^{\sigma_1}\wedge\cdots\wedge dx^{\sigma_k})
=\sum_{j=1}^n\frac{\partial a_\sigma}{\partial x_j}\,dx^j\wedge dx^{\sigma_1}\wedge\cdots\wedge dx^{\sigma_k}.
$$
It satisfies (in general) the identity $d^{k+1}\circ d^k=0$
and the Leibniz rule $d(\omega\wedge \mu)=(d\omega)\wedge \mu +(-1)^k\omega\wedge(d\mu)$
if $\omega$ is a $k$-form.

  The definition of the form integral
requires that the manifold $\Omega$ be oriented.  In this case we can define $\int_\Omega \omega\in\R$
for $\omega$ an $n$-form with $n=\dim\Omega$.  The integral changes sign if the orientation of the
manifold is reversed.

Finally, if $F:\Omega\to\Omega'$ is a differentiable map, then the pullback
$F^*$ takes a $k$-form on $\Omega'$ to one on $\Omega$ by
$$
(F^*\omega)_x(v_1,\ldots,v_k)=\omega_{F(x)}(dF_xv_1,\ldots,dF_xv_k), \quad x\in\Omega,\ v_1,\ldots,v_k\in T_x\Omega.
$$
The pullback respects the operations of wedge product, exterior derivative, and form integral:
$$
F^*(\omega\wedge\mu)=(F^*\omega)\wedge(F^*\mu),\quad
F^*(d\omega)=d(F^*\omega), \quad \int_\Omega F^*\omega=\int_{\Omega'}\omega,
$$
for $\omega$ and $\mu$ differential forms on $\Omega'$.  The last relation requires that $F$ be a diffeomorphism
of $\Omega$ with $\Omega'$ which preserves orientation.

An important special case of pullback is when $F$ is the inclusion of a submanifold $\Omega$ into a larger manifold $\Omega'$.
In this case the pullback is the trace operator taking a $k$-form on $\Omega'$ to a $k$-form
on the submanifold $\Omega$.  All these operations combine elegantly into Stokes' theorem, which says that, under
minimal hypothesis on the smoothness of the differential $(n-1)$-form $\omega$ and the $n$-manifold $\Omega$,
$$
\int_{\partial\Omega}\tr \omega = \int_\Omega d\omega.
$$

If $V$ is an inner product space, then there is a natural inner product on $\Alt^k V$.  Thus for
a Riemannian manifold, such as any manifold embedded in $\R^n$, the inner product $\<\omega_x,\mu_x\>\in\R$
is defined for any $k$-forms $\omega$, $\mu$ and any $x\in\Omega$.  An oriented Riemannian manifold
also has a unique volume form, $\vol$, a differential $n$-form which at each point assigns the value $1$ to a positively
oriented orthonormal basis for the tangent space at that point.  For a subdomain of $\R^n$ the volume form
is the constant $n$-form with the value $dx^1\wedge\cdots\wedge dx^n$ at each point. Combining these notions,
we see that on any oriented Riemannian manifold we may define the $L^2$-inner product of $k$-forms:
$$
\<\omega,\mu\>_{L^2\Lambda^k(\Omega)} = \int_\Omega\<\omega_x,\mu_x\>\,\vol.
$$
The space $L^2\Lambda^k$ is of course the space of $k$-forms for which
$\|\omega\|_{L^2\Lambda^k}:=\sqrt{\<\omega,\omega\>_{L^2\Lambda^k}}<\infty$,
and then $H\Lambda^k$ is defined as in \eqref{hlambda}.

\section{Families of finite element differential forms on simplicial meshes}\label{sec:simp}
Our goal now is to create finite element subspaces of the spaces $H\Lambda^k$
which fit together to yield a subcomplex with commuting projections.
In this section the spaces will be constructed for a simplicial
triangulation $\T_h$ of the domain $\Omega\subset\R^n$.  Thus,
for a simplex $T$, we must specify
a space $V(T)$ of polynomial differential forms
and a set of degrees of freedom for it.

\subsection{The polynomial space $\P_r\Lambda^k$}
An obvious choice for $V(T)$ is the space
$$
\P_r\Lambda^k(T) = \{\sum_{\sigma\in\Sigma(k,n)} p_\sigma\,dx^\sigma\,|\, p_\sigma\in
\P_r(T)\,\},
$$
of a differential $k$-forms with polynomial coefficients of degree at most $r$.  It is
easy to compute its dimension:
\begin{equation}\label{dimPr}
\dim\P_r\Lambda^k(T) = \#\Sigma(k,n)\x \dim\P_r(T) = \binom{n}{k}\binom{n+r}{n} =
\binom{n+r}{n-k}\binom{r+k}{r}.
\end{equation}
Note that $d\P_r\Lambda^k\subset \P_{r-1}\Lambda^{k+1}$, i.e., the exterior derivative
lowers the polynomial degree at the same time as it raises the form degree.  Therefore,
for each $r$ we have a subcomplex of the de~Rham complex:
\begin{equation}\label{pdR}
\P_r\Lambda^0 \xrightarrow{d} \P_{r-1}\Lambda^1 \xrightarrow{d} 
\cdots \xrightarrow{d} \P_{r-n}\Lambda^n \to 0.
\end{equation}
This complex is exact (we have left off the initial $0$ since the first map,
$d=\grad$ acting on $\P_r\Lambda^0$ has a $1$-dimensional kernel, consisting of
the constant functions).  That is, if $\omega\in\P_s\Lambda^k$ and $d\omega=0$
then $\omega = d\mu$ for some $\mu\in\P_{s+1}\Lambda^{k-1}$.  We prove
this in Corollary~\ref{exact} below, using
an elementary but powerful tool called the \emph{Koszul complex}.  The
same tool will also be used to define the degrees of freedom for $\P_r\Lambda^k(T)$,
and to define an alternative space
of shape functions.

\subsection{The Koszul complex}
For a domain in $\Omega\subset\R^n$
(but not a general manifold), the identity map may be viewed as a vector field.
It assigns to an arbitrary point $x\in\Omega\subset\R^n$ the point itself viewed
as a vector in $\R^n$ and so an element of the tangent space $T_x\Omega$.
Contracting a $k$-form $\omega$ with this identity vector field gives a $(k-1)$-form $\kappa\omega$:
$$
(\kappa\omega)_x(v_1,\ldots,v_{k-1})=\omega_x(x,v_1,\ldots,v_{k-1}),
\quad x\in\Omega,\ v_1,\ldots,v_{k-1}\in\R^n.
$$
Since $\omega_x$ is skew-symmetric, $\kappa\kappa\omega=0$, that is, $\kappa$ is a differential.
It satisfies a Leibniz rule:
$$
\kappa(\omega\wedge\mu) = (\kappa\omega)\wedge\mu + (-1)^k\omega\wedge(\kappa\mu),
$$
for a $k$-form $\omega$ and a second form $\mu$.  In particular $\kappa(f\omega)=f\kappa\omega$
if $f$ is a function.  Also $\kappa dx^i=x^i$.  These properties fully determine $\kappa$.
Thus
$$
\kappa (dx^i\wedge dx^j) = x^i\,dx^j - x^j\,dx^i, \quad
\kappa (dx^i\wedge dx^j\wedge dx^k) = x^i\, dx^j\wedge dx^k-x^j\,dx^i\wedge dx^k+x^k\,dx^i\wedge dx^j,
$$
and so forth.
If we identify $1$-forms with vector fields, then $\kappa$ corresponds to the dot product
of the vector field with $x$ (or, more properly, with the identity vector field).
On $2$-forms in 3-D, $\kappa$ is the cross product with $x$, and on $3$-forms it is the product
of a scalar field with $x$ to get a vector field.

The Koszul differential $\kappa$ maps the space $\P_r\Lambda^k$ of differential $k$-forms
with coefficients in $\P_r(\Omega)$ to $\P_{r+1}\Lambda^{k-1}$, exactly
the reverse of $d$.  Thus both $\kappa d$ and $d\kappa$ map $\P_r\Lambda^k$ to itself.
The following theorem points to an intimate relation between $\kappa$ and $d$,
called the \emph{homotopy formula}.  In it
we write $\H_r\Lambda^k$ for the $k$-forms with \emph{homogeneous} polynomial
coefficients of degree $r$.
\begin{thm}[Homotopy formula]\label{homotopy}
 $$
(\kappa d + d\kappa) \omega = (k+r)\omega, \quad \omega\in\H_r\Lambda^k.
$$
\end{thm}
\begin{proof}[Remarks on the proof]
 The case $k=0$ is Euler's identity $x\cdot \grad p = r\, p$ for $p$ a homogeneous
polynomial of degree $r$.  Using it, we can verify the theorem by direct computation.
Alternatively, one may use Cartan's homotopy formula from differential geometry.
For details on both proofs, see Theorem 3.1 of \cite{acta}.
\end{proof}
\begin{cor}\label{exact}
 The polynomial de Rham complex \eqref{pdR}
and the Koszul complex
$$
0\to \P_{r-n}\Lambda^n \xrightarrow{\kappa}\P_{r-n+1}\Lambda^{n-1}
\xrightarrow{\kappa} \cdots \xrightarrow{\kappa} \P_r\Lambda^0
$$
are both exact.
\end{cor}
\begin{proof}
For the de~Rham complex, it suffices to establish exactness of the homogeneous polynomial de Rham complex
$$
\H_r\Lambda^0 \xrightarrow{d} \H_{r-1}\Lambda^1 \xrightarrow{d} 
\cdots \xrightarrow{d} \H_{r-n}\Lambda^n \to 0,
$$
since then we can then just sum to get the result.  We must show that
if $\omega\in \H_s\Lambda^k$ and $d\omega=0$ then $\omega$
is in the range of $d$.  Indeed, by the homotopy formula
$$
\omega = (s+k)^{-1}(d\kappa + \kappa d)\omega = (s+k)^{-1}d\kappa\omega.
$$
A similar proof holds for the Koszul complex.
\end{proof}
Another important consequence is a direct sum decomposition:
\begin{cor} For $r\ge1$, $0\le k\le n$, 
 \begin{equation}\label{Hdirectsum}
\H_r\Lambda^k = \kappa\H_{r-1}\Lambda^{k+1} \oplus
d\H_{r+1}\Lambda^{k-1}.
\end{equation}
\end{cor}
\begin{proof}
 By the homotopy formula, any element of $\H_r\Lambda^k$ belongs
to $\kappa\H_{r-1}\Lambda^{k+1} +
d\H_{r+1}\Lambda^{k-1}$.  Moreover the intersection of these two
spaces is zero, since if $\omega$ belongs to the intersection,
then $d\omega=0$, $\kappa\omega=0$, so $\omega=0$ by the homotopy formula.
\end{proof}

\subsection{The polynomial space $\P_r^-\Lambda^k$}
We now define a second space of polynomial differential forms
which can be used as shape functions.
We have
$$
\P_r\Lambda^k = \P_{r-1}\Lambda^k \oplus \H_r\Lambda^k =
\P_{r-1}\Lambda^k \oplus\kappa\H_{r-1}\Lambda^{k+1} \oplus
d\H_{r+1}\Lambda^{k-1}.
$$
If we drop the last summand, we get a space intermediate between $\P_{r-1}\Lambda^k$
and $\P_r\Lambda^k$:
\begin{equation}\label{defPrm}
\P_r^-\Lambda^k:= \P_{r-1}\Lambda^k +\kappa\H_{r-1}\Lambda^{k+1}.
\end{equation}
Note that $\P_r^-\Lambda^0=\P_r\Lambda^0$ and
$\P_r^-\Lambda^n=\P_{r-1}\Lambda^n$, but for $0<k<n$, $\P_r^-\Lambda^k$
is contained strictly between $\P_{r-1}\Lambda^k$ and $\P_r\Lambda^k$.
We may compute the dimension of $\kappa\H_r\Lambda^k$, using the exactness of the Koszul
complex and induction (see \cite[Theorem 3.3]{acta}).
This then yields
a formula for the dimension of $\P_r^-\Lambda^k$:
$$
\dim \P_r^-\Lambda^k = \binom{n+r}{n-k}\binom{r+k-1}{k}.
$$
Comparing this with \eqref{dimPr}, we have
$$
\dim \P_r^-\Lambda^k = \frac{r}{r+k}\dim\P_r\Lambda^k
$$
(showing again that the spaces coincide for $0$-forms).

Now
$$
d\P_r^-\Lambda^k\subset d\P_r\Lambda^k\subset \P_{r-1}\Lambda^{k+1}\subset\P_r^-\Lambda^{k+1},
$$
so we obtain another subcomplex of the de~Rham complex:
\begin{equation}\label{pmdR}
\P_r^-\Lambda^0\xrightarrow{d} \P_r^-\Lambda^1\xrightarrow{d}
\cdots\xrightarrow{d}\P_r^-\Lambda^n\to 0.
\end{equation}
Note that, in contrast to \eqref{pdR}, in this complex the degree $r$ is
held constant.
However, like \eqref{pdR}, the complex \eqref{pmdR} is exact.  Indeed,
$$
d\P_r^-\Lambda^k=d(\P_r^-\Lambda^k+d\P_{r+1}\Lambda^{k-1})=
d\P_r\Lambda^k=
\Null(d|{\P_{r-1}\Lambda^{k+1}})=\Null(d|{\P_r^-\Lambda^{k+1}}),
$$
where the penultimate equality follows from
Corollary~\ref{exact} and the last equality is a consequence of the definition
\eqref{defPrm} and the homotopy formula
Theorem~\ref{homotopy}.

\subsection{The $\P_r^-\Lambda^k(\T_h)$ family of finite element differential forms}
Let $r\ge 1$, $0\le k\le n$, and let $\T_h$ be a simplicial mesh of $\Omega\subset\R^n$.
We define a finite element subspace $\P_r^-\Lambda^k(\T_h)$ of $H\Lambda^k(\Omega)$.
As shape functions on a simplex $T\in\T_h$ we take $V(T)=\P_r^-\Lambda^k(T)$.
As degrees of freedom we take
\begin{equation}\label{Prmdofs}
u\in\P_r^-\Lambda^k(T)\mapsto \int_f (\tr_f u)\wedge q, \quad q\in\P_{r+k-d-1}\Lambda^{d-k}(f), \quad
f\in\Delta_d(T), \ d\ge k.
\end{equation}
Note that, in the case $k=0$, $V(T)=\P_r(T)$ and \eqref{Prmdofs} coincides with \eqref{lagdofs},
so the space $\P_r^-\Lambda^k(\T_h)$ generalizes the Lagrange finite elements to
differential forms of arbitrary form degree.  We shall prove unisolvence for arbitrary
polynomial degree, form degree, and space dimension at once.  The proof will use the
following lemma, which is proved via a simple construction using barycentric coordinates.
\begin{lemma}
 Let $r\ge 1$, $0\le k\le n$, and let $T$ be an $n$-simplex.
If $u\in \0\P_{r-1}\Lambda^k(T)$ and 
\begin{equation}\label{vhyp}
 \int_T u\wedge q=0, \quad q\in \P_{r+k-n-1}\Lambda^{n-k}(T), 
\end{equation}
then $u\equiv 0$.
\end{lemma}
\begin{proof}
Any element of $\P_{r-1}\Lambda^k(T)$ can be written in terms of barycentric
coordinates as
$$
u=\sum_{\sigma\in\Sigma(k,n)} u_\sigma 
 d\lambda_{\sigma_1}\wedge\cdots\wedge d\lambda_{\sigma_k}, \quad u_\sigma\in\P_{r-1}(T).
$$
Now let $1\le i\le n$, and consider the trace of $u$ on the face given by $\lambda_i=0$.
By the assumption that $u\in\0\P_{r-1}\Lambda^k(T)$, the trace vanishes.
This implies that $\lambda_i$ divides $u_\sigma$ for any $\sigma\in\Sigma(k,n)$
whose range does not contain $i$.  Thus
$$
u_\sigma=p_\sigma\lambda_{\sigma^*_1}\cdots\lambda_{\sigma^*_{n-k}}
\text{ for some $p_\sigma\in\P_{r+k-n-1}(T)$},
$$
where $\sigma^*\in\Sigma(n-k,n)$ is the increasing sequence complementary to $\sigma$.
Thus
$$
u=\sum_{\sigma\in\Sigma(k,n)} p_\sigma \lambda_{\sigma^*_1}\cdots\lambda_{\sigma^*_{n-k}}
 d\lambda_{\sigma_1}\wedge\cdots\wedge d\lambda_{\sigma_k},
\quad p_\sigma\in\P_{r+k-n-1}(T).
$$
Choosing
$$
q= \sum_{\sigma\in \Sigma(k,n)} (-1)^{\operatorname{sign}(\sigma,\sigma^*)}
 p_\sigma d\lambda_{\sigma^*_1}\wedge\cdots\wedge d\lambda_{\sigma^*_{n-k}}
$$
in \eqref{vhyp}, we get
$$
0=\int_T u\wedge q = \int_T \sum_{\sigma\in\Sigma(k,n)} p_\sigma^2 
\lambda_{\sigma^*_1}\cdots\lambda_{\sigma^*_{n-k}}d\lambda_1\wedge\cdots\wedge d\lambda_n.
$$
However, the $\lambda_i$ are positive on the interior of $T$ and the $n$-form
$d\lambda_1\wedge\cdots\wedge d\lambda_n$ is a nonzero multiple of the volume form.  Thus each
$p_\sigma$ must vanish, and so $u$ vanishes.
\end{proof}

\begin{thm}[Unisolvence for $\P_r^-\Lambda^k(\T_h)$]\label{unisolv}
 For any $r\ge1$, $0\le k\le n$, and $n$-simplex $T$, the degrees of freedom \eqref{Prmdofs}
are unisolvent for $V(T)=\P_r^-\Lambda^k(T)$.
\end{thm}
\begin{proof}
 First we do the dimension count.  The number of degrees of freedom is at most
\begin{align*}
\sum_{d\ge k} \#\Delta_d(T) \dim\P_{r+k-d-1}\Lambda^k(\R^d)
&= \sum_{d\ge k} \binom{n+1}{d+1}\binom{r+k-1}{d}\binom{d}{k}\\
&= \sum_{j\ge 0} \binom{n+1}{j+k+1}\binom{r+k-1}{j+k}\binom{j+k}{j}.
\end{align*}
Simplifying with the binomial identities,
$$
\binom{a}{b}\binom{b}{c}=\binom{a}{c}\binom{a-c}{a-b}, \quad
\sum_{j\ge 0} \binom{a}{b+j}\binom{c}{j}=\binom{a+c}{a-b},
$$
the right-hand side becomes
$$
 \binom{r+n}{r+k}\binom{r+k-1}{k} = \dim\P_r^-\Lambda^k(T).
$$

It remains to show that if $u\in \P_r^-\Lambda^k(T)$ and the degrees of freedom in \eqref{Prmdofs}
vanish, then $u$ vanishes.  Since $\tr_f\P_r^-\Lambda^k(T)=\P_r^-\Lambda^k(f)$, we
may use induction on dimension to conclude that $\tr_f u$ vanishes on each facet $f$, so
$u\in\0\P_r^-\Lambda^k(T)$.  Therefore $du\in\0\P_{r-1}\Lambda^{k+1}(T)$.  Moreover,
$$
\int_T du\wedge p = \pm \int_T u\wedge dp =0, \quad p\in\P_{r+k-n}\Lambda^{n-k-1}(T),
$$
where the first equality comes from Stoke's theorem and the Leibniz rule, and the second
from the hypothesis that the degrees of freedom for $u$ vanish.  We may now apply
the lemma (with $k$ replaced by $k+1$) to $du$ to conclude that $du$ vanishes.
But the homotopy formula implies that for $u\in\P_r^-\Lambda^k$ with $du=0$, $u\in\P_{r-1}\Lambda^k$.
Using the interior degrees of freedom from \eqref{Prmdofs}, we may apply the lemma
to $u$, to conclude that $u$ vanishes.
\end{proof}

It is easy to check that the degrees of freedom imply single-valuedness of the traces
of elements of $\P_r^-\Lambda^k(\T_h)$, so that they indeed belong to $H\Lambda^k$.
Moreover, it is easy to see that
the complex \eqref{pmdR} involving the shape functions, leads to a finite element subcomplex
of the $L^2$ de~Rham complex on $\Omega$:
\begin{equation*}
\P_r^-\Lambda^0(\T_h)\xrightarrow{d} \P_r^-\Lambda^1\xrightarrow{d}(\T_h)
\cdots\xrightarrow{d}\P_r^-\Lambda^n(\T_h).
\end{equation*}
Using the degrees of freedom to define projection operators $\pi^k_h$ into
$\P_r^-\Lambda^k(\T_h)$ (the domain of $\pi^k_h$ consists of all continuous $k$-forms in
$H\Lambda^k(\Omega)$), we obtain projections that commute with $d$ (this can be verified
using Stokes' theorem), which is crucial to the analysis of the element via FEEC.

\subsection{The $\P_r\Lambda^k(\T_h)$ family of finite element differential forms}
We may also use the full polynomial space $\P_r\Lambda^k(T)$ as shape functions
for a finite element space.  The corresponding degrees of freedom are
\begin{equation}\label{Prdofs}
u\in\P_r\Lambda^k(T)\mapsto \int_f (\tr_f u)\wedge q, \quad q\in\P_{r+k-d}^-\Lambda^{d-k}(f), \quad
f\in\Delta_d(T), \ d\ge k.
\end{equation}
Note that in this case the degrees of freedom involve $\P_r^-$ spaces, defined
through the Koszul complex.
The analysis of these spaces is very parallel to that of the last subsection,
and we will not carry it out here.  Again, we obtain unisolvence, and a finite element
subcomplex of the de~Rham complex
$$
\P_r\Lambda^0(\T_h) \xrightarrow{d} \P_{r-1}\Lambda^1(\T_h) \xrightarrow{d} 
\cdots \xrightarrow{d} \P_{r-n}\Lambda^n(\T_h),
$$
which admits a commuting projection defined via the degrees of freedom.

\subsection{Historical notes}
In the case $k=0$, the two shape function spaces $\P_r^-\Lambda^k$ and $\P_r\Lambda^k$
coincide, as do the spaces $\P_{r-d-1}\Lambda^{d-k}(f)$ and $\P_{r-d}^-\Lambda^{d-k}(f)$,
$f\in\Delta_d(T)$, entering \eqref{Prmdofs} and \eqref{Prdofs}.  Thus the two
finite element families coincide for $0$-forms, and provide two distinct generalizations
of the Lagrange elements to differential forms of higher degree.

In $n$ dimensions, $n$-forms may be viewed as scalar functions and the space
$H\Lambda^n(\Omega)$ just corresponds to $L^2(\Omega)$.
The finite element subspace $\P_r\Lambda^n(\T_h)$ is simply the space of all
piecewise polynomial functions of degree $r$, with no interelement continuity required.
The space $\P_r^-\Lambda^n(\T_h)$ coincides with $\P_{r-1}\Lambda^n(\T_h)$.

In two dimensions, the remaining spaces $\P_r^-\Lambda^1(\T_h)$ and $\P_r\Lambda^1(\T_h)$
can be identified, via vector proxies, with the Raviart--Thomas spaces \cite{Raviart-Thomas}
and the Brezzi--Douglas--Marini spaces \cite{Brezzi-Douglas-Marini}.  In three dimensions,
the $\P_r^-\Lambda^1(\T_h)$ and $\P_r^-\Lambda^2(\T_h)$ spaces are the finite element
subspaces of $H(\curl,\Omega)$ and $H(\div,\Omega)$, respectively, called
the N\'ed\'elec edge and face elements of the first kind \cite{Nedelec1}.
The spaces  $\P_r\Lambda^1(\T_h)$ and $\P_r\Lambda^2(\T_h)$ are the
N\'ed\'elec edge and face elements of the second kind \cite{Nedelec2}.
Diagrams for the two-dimensional and three-dimensional elements are shown in Figures~\ref{fg:simp2d}
and Figures~\ref{fg:simp3d}.

\begin{figure}[p]
\begin{center}
\begin{tabular}{cccc}
$\P_r^-\Lambda^k$: & $k=0$ & $k=1$ & $k=2$
\\[.15in]
\raise.5in\hbox{$r=1$}
  & \includegraphics[width=1.25in]{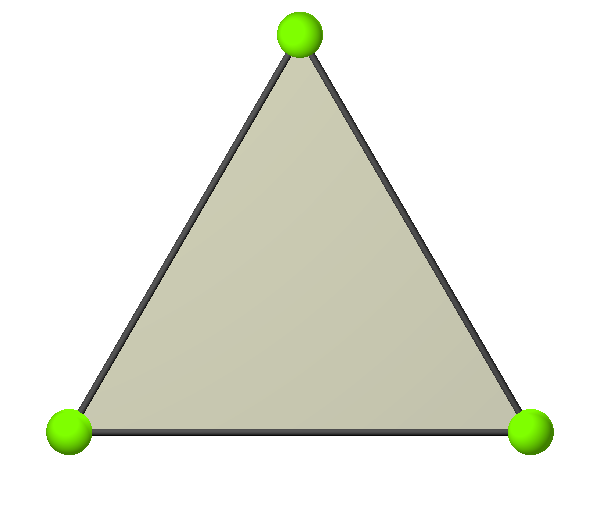}
  & \includegraphics[width=1.25in]{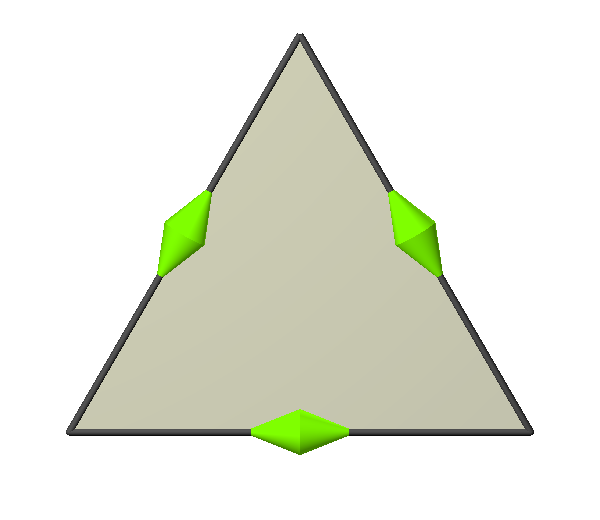}
  & \includegraphics[width=1.25in]{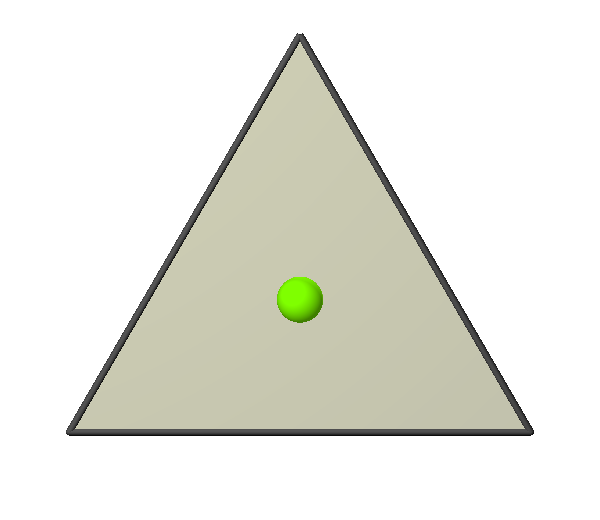}
\\
\raise.5in\hbox{$r=2$}
  & \includegraphics[width=1.25in]{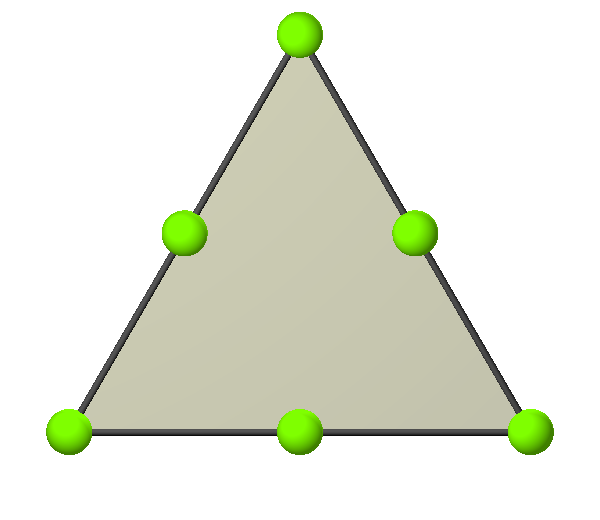}
  & \includegraphics[width=1.25in]{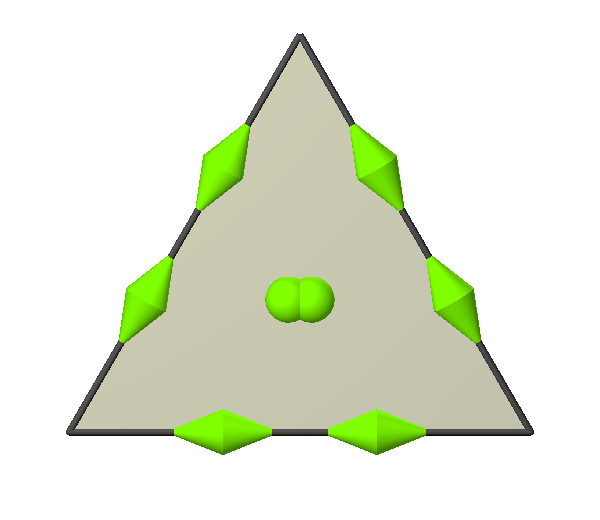}
  & \includegraphics[width=1.25in]{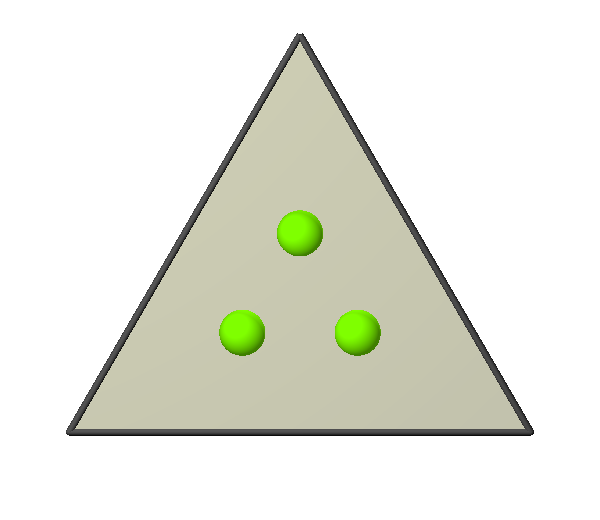}
\\
\raise.5in\hbox{$r=3$}
  & \includegraphics[width=1.25in]{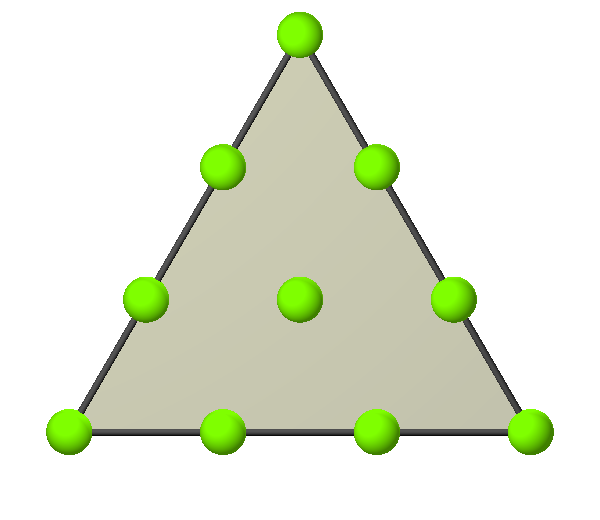}
  & \includegraphics[width=1.25in]{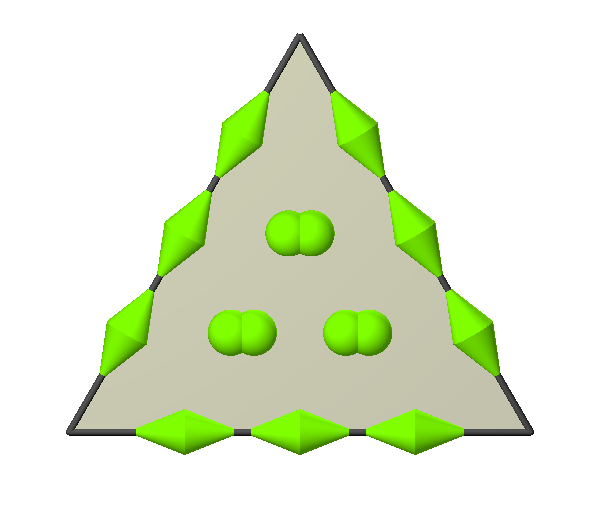}
  & \includegraphics[width=1.25in]{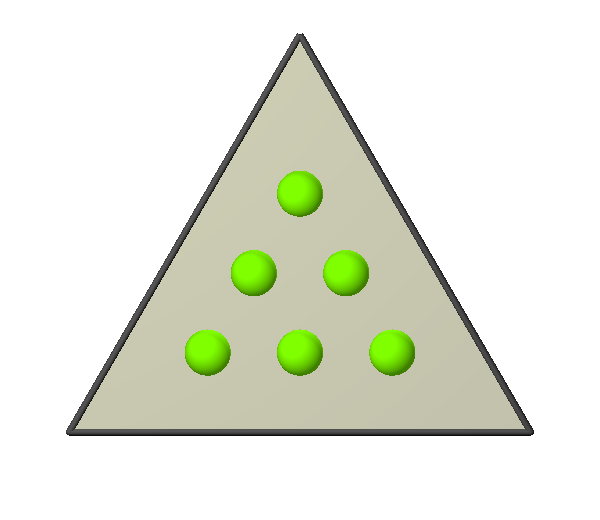}
\\[.25in]
$\P_r\Lambda^k$: & $k=0$ & $k=1$ & $k=2$
\\[.15in]
\raise.5in\hbox{$r=1$}
  & \includegraphics[width=1.25in]{element_diagrams/simplicial-elts/PminusLambda/lagrange2d/lagrange2d1.png}
  & \includegraphics[width=1.25in]{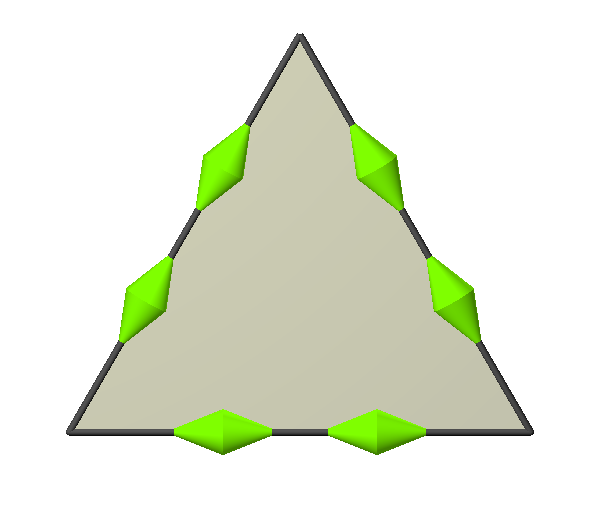}
  & \includegraphics[width=1.25in]{element_diagrams/simplicial-elts/PminusLambda/dg2d/dg2d1.png}
\\
\raise.5in\hbox{$r=2$}
  & \includegraphics[width=1.25in]{element_diagrams/simplicial-elts/PminusLambda/lagrange2d/lagrange2d2.png}
  & \includegraphics[width=1.25in]{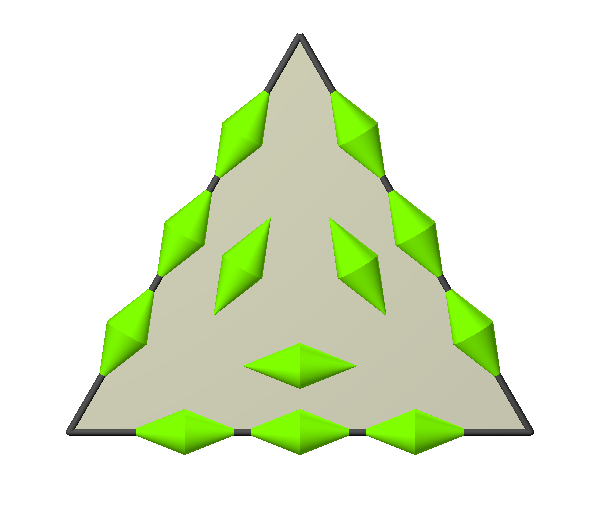}
  & \includegraphics[width=1.25in]{element_diagrams/simplicial-elts/PminusLambda/dg2d/dg2d2.png}
\\
\raise.5in\hbox{$r=3$}
  & \includegraphics[width=1.25in]{element_diagrams/simplicial-elts/PminusLambda/lagrange2d/lagrange2d3.png}
  & \includegraphics[width=1.25in]{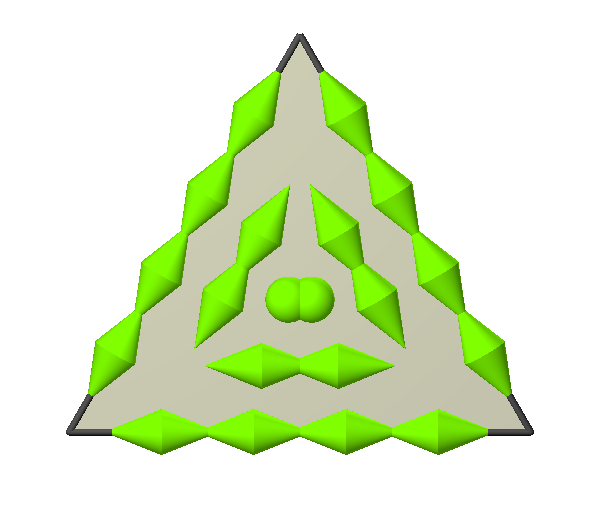}
  & \includegraphics[width=1.25in]{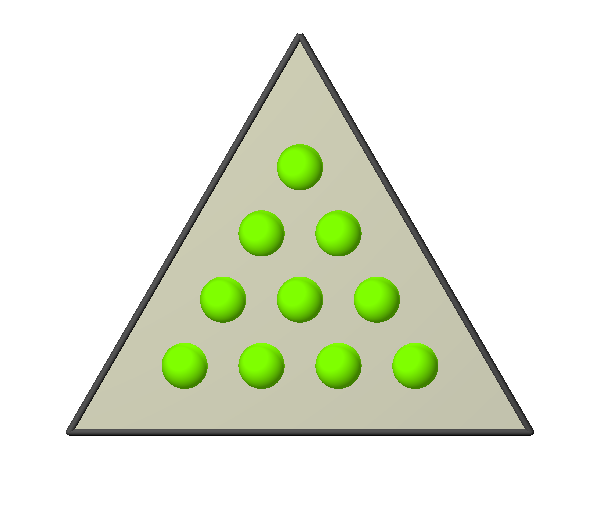}
 \end{tabular}
\end{center}
Legend:
\raise-1pt\hbox{\includegraphics[height=.13in]{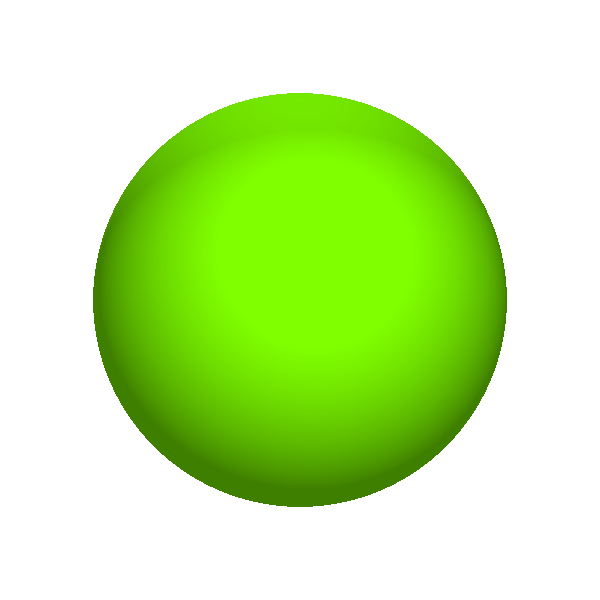}}~--~scalar (1 DOF);
\raise-4pt\hbox{\includegraphics[height=.187in]{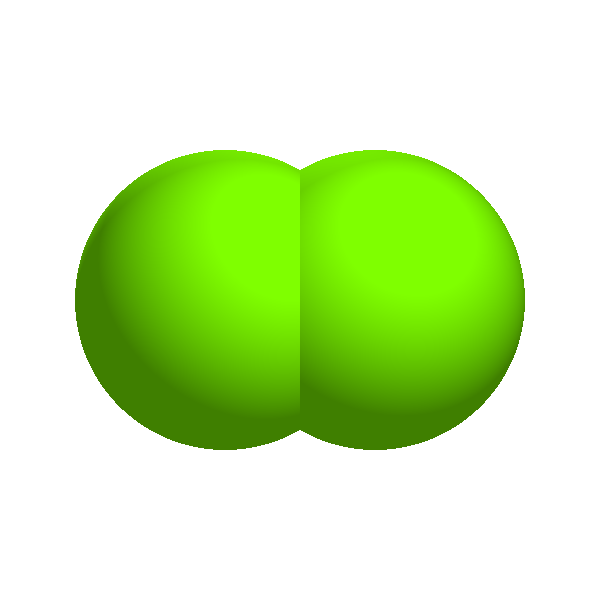}}~--~vector field (2 DOFs);\\\leavevmode
\raise-5pt\hbox{\includegraphics[height=.245in]{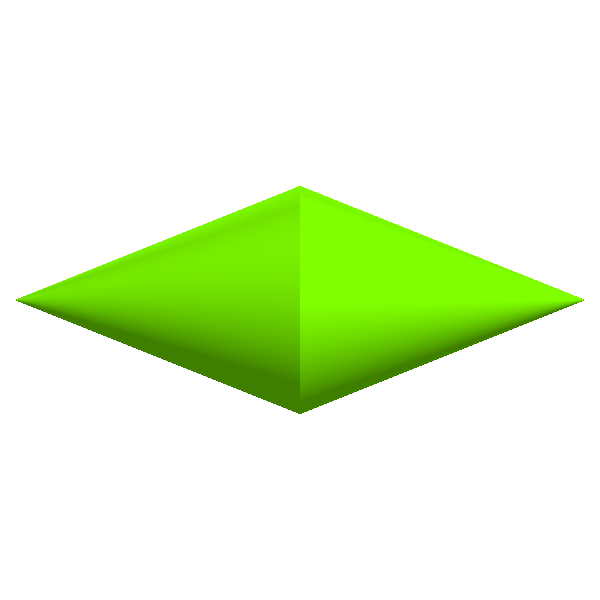}}~--~tangential component along edge (1 DOF).
\caption{The $\P_r^-\Lambda^k(\T_h)$ and $\P_r\Lambda^k(\T_h)$ spaces in two dimensions.}\label{fg:simp2d}
\end{figure}

\begin{figure}[p]
\begin{center}
\begin{tabular}{ccccc}
$\P_r^-\Lambda^k$: & $k=0$ & $k=1$ & $k=2$ & $k=3$
\\[.15in]
\raise.5in\hbox{$r=1$}
  & \includegraphics[width=1.25in]{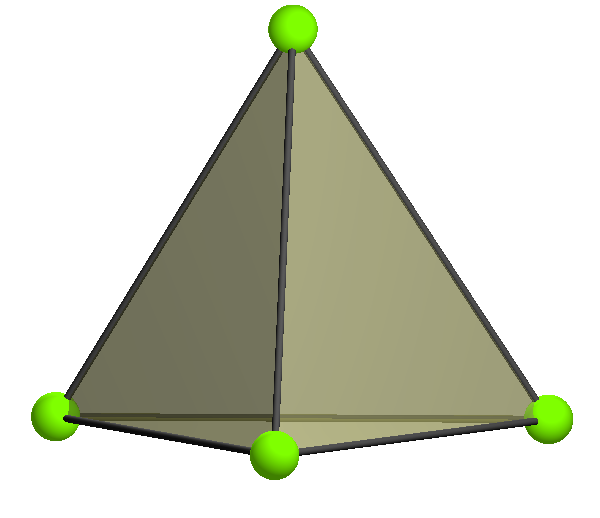}
  & \includegraphics[width=1.25in]{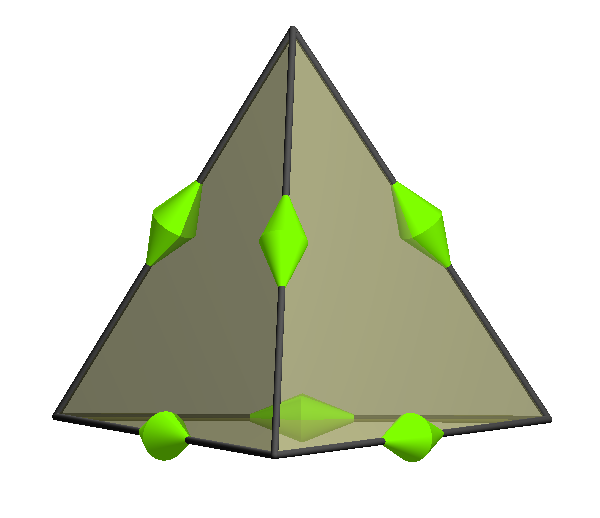}
  & \includegraphics[width=1.25in]{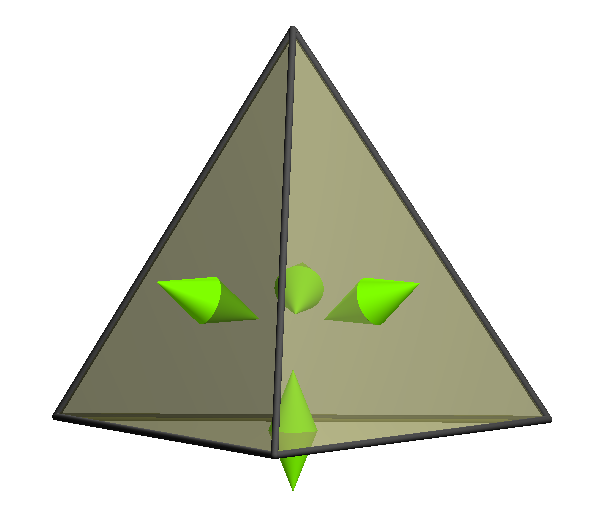}
  & \includegraphics[width=1.25in]{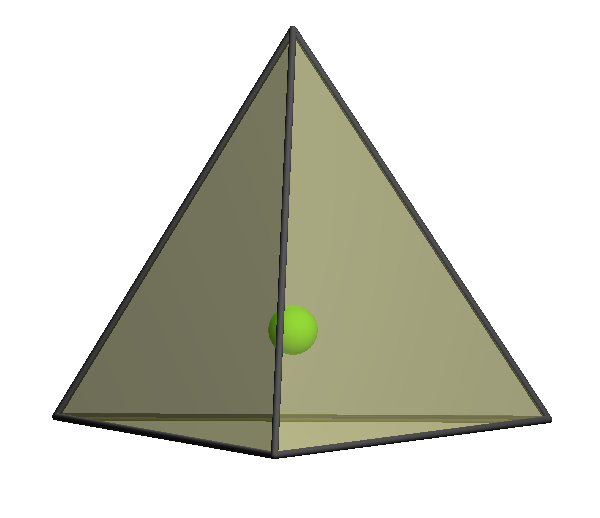}
\\
\raise.5in\hbox{$r=2$}
  & \includegraphics[width=1.25in]{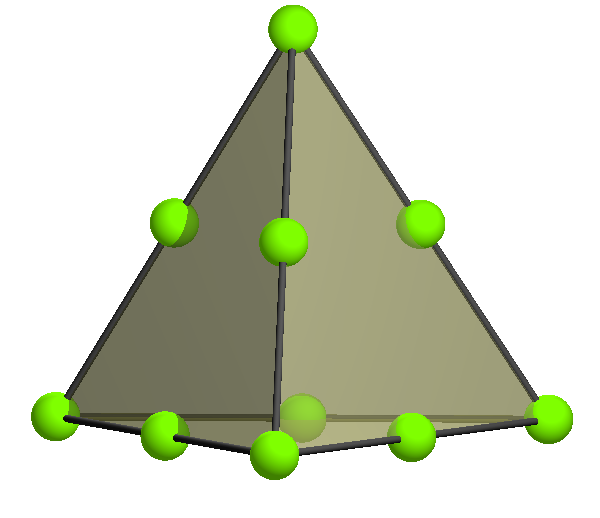}
  & \includegraphics[width=1.25in]{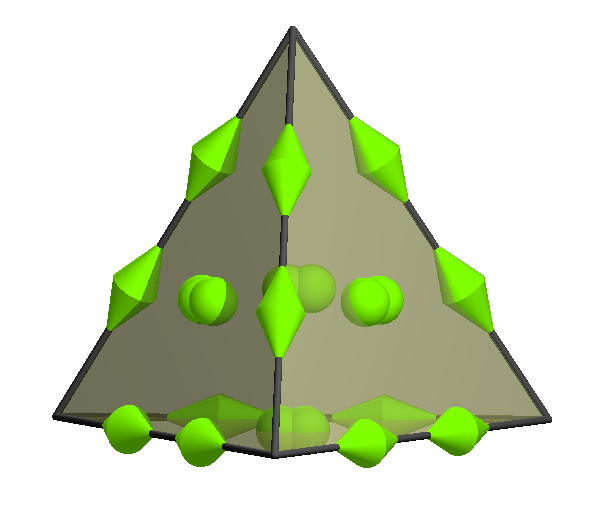}
  & \includegraphics[width=1.25in]{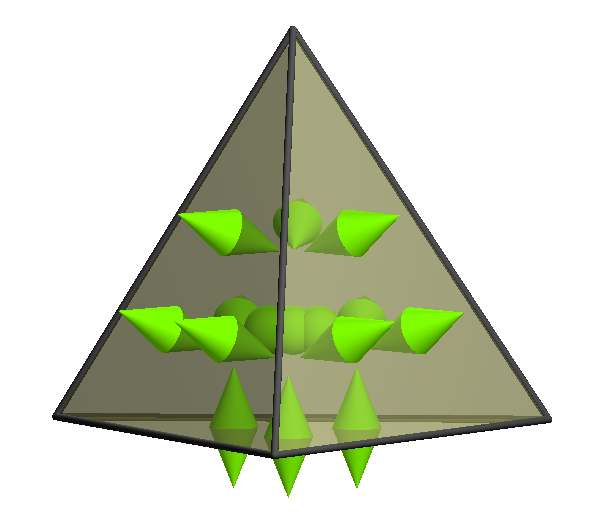}
  & \includegraphics[width=1.25in]{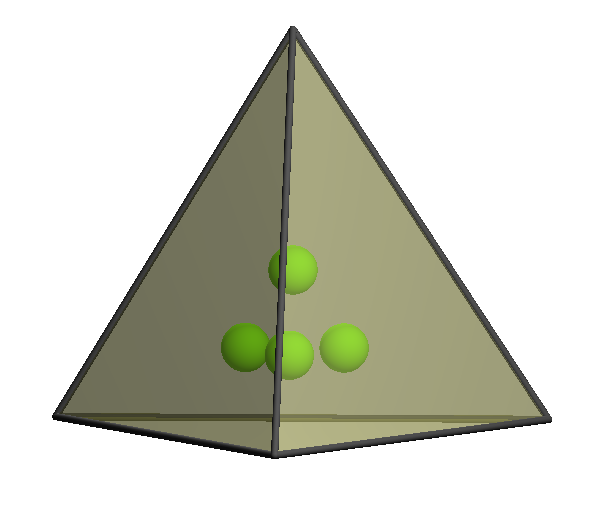}
\\
\raise.5in\hbox{$r=3$}
  & \includegraphics[width=1.25in]{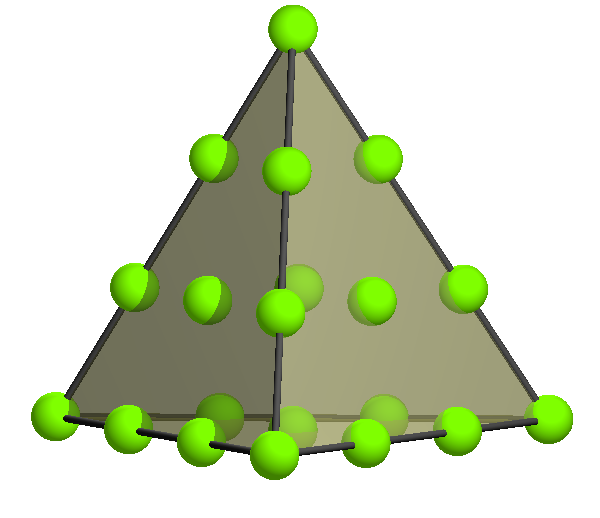}
  & \includegraphics[width=1.25in]{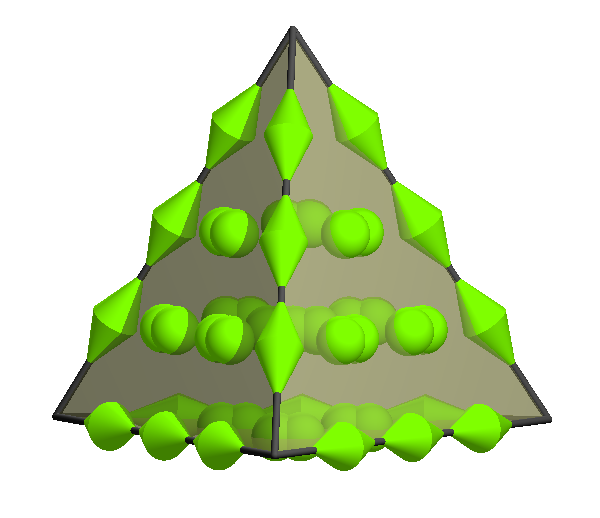}
  & \includegraphics[width=1.25in]{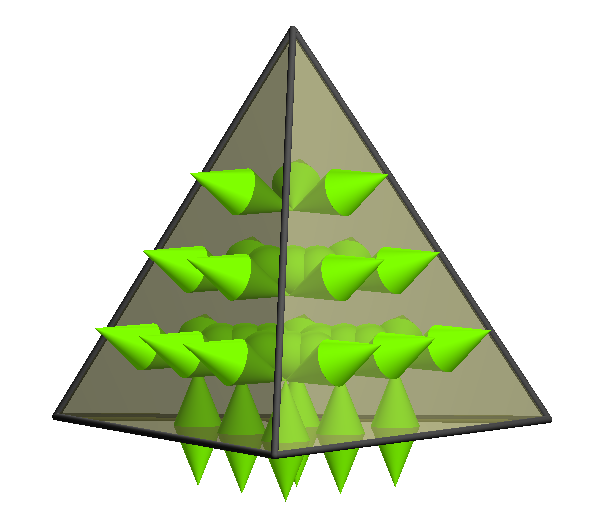}
  & \includegraphics[width=1.25in]{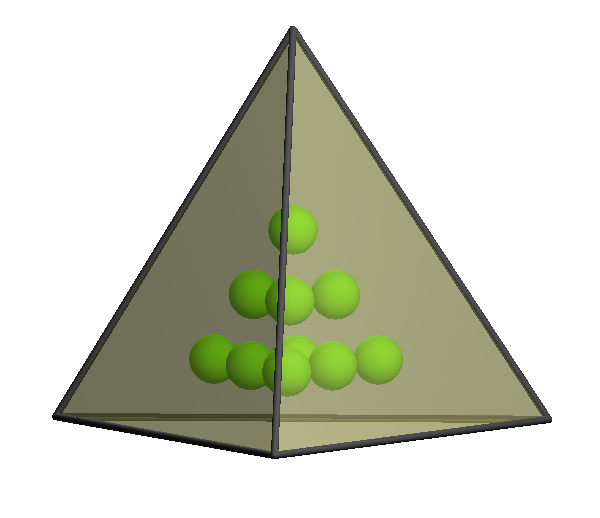}
\\[.25in]
$\P_r\Lambda^k$: & $k=0$ & $k=1$ & $k=2$ & $k=3$
\\[.15in]
\raise.5in\hbox{$r=1$}
  & \includegraphics[width=1.25in]{element_diagrams/simplicial-elts/PminusLambda/lagrange3d/lagrange3d1.png}
  & \includegraphics[width=1.25in]{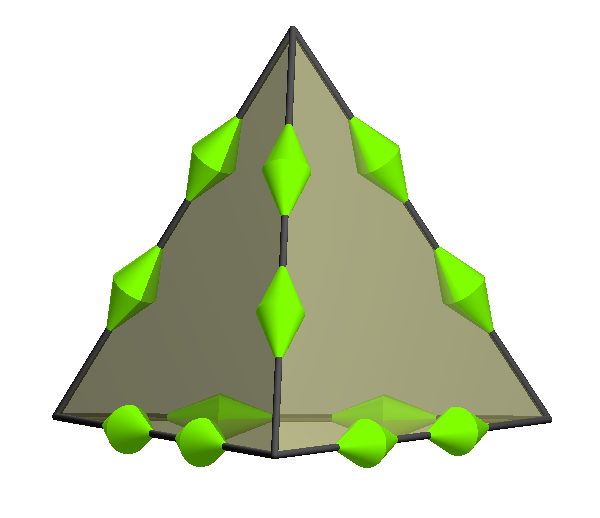}
  & \includegraphics[width=1.25in]{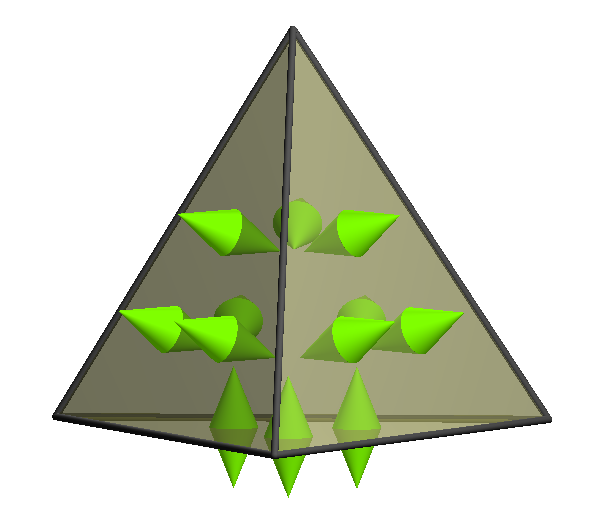}
  & \includegraphics[width=1.25in]{element_diagrams/simplicial-elts/PminusLambda/dg3d/dg3d1.png}
\\
\raise.5in\hbox{$r=2$}
  & \includegraphics[width=1.25in]{element_diagrams/simplicial-elts/PminusLambda/lagrange3d/lagrange3d2.png}
  & \includegraphics[width=1.25in]{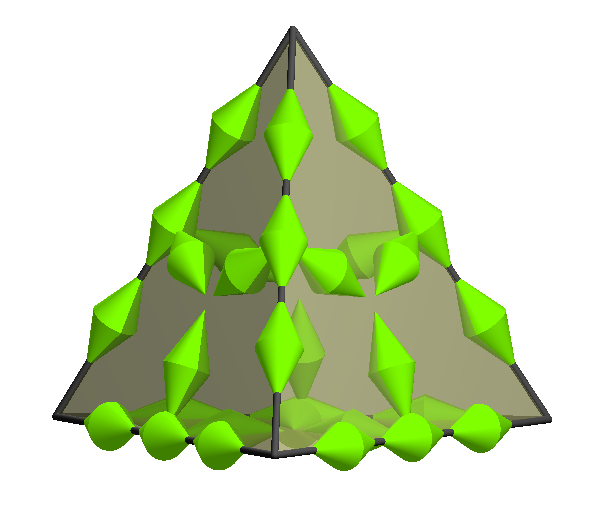}
  & \includegraphics[width=1.25in]{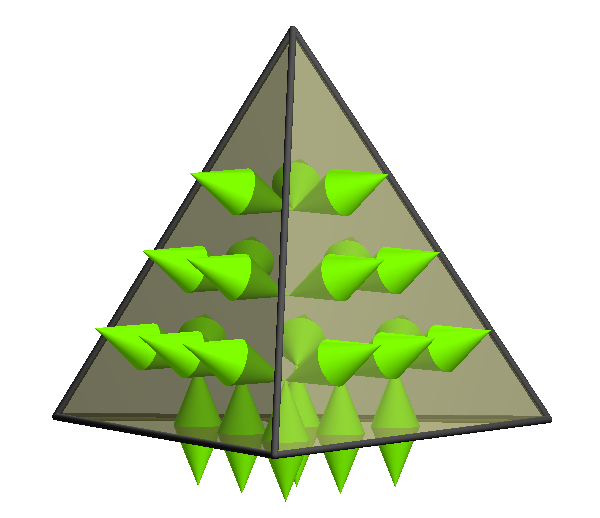}
  & \includegraphics[width=1.25in]{element_diagrams/simplicial-elts/PminusLambda/dg3d/dg3d2.png}
\\
\raise.5in\hbox{$r=3$}
  & \includegraphics[width=1.25in]{element_diagrams/simplicial-elts/PminusLambda/lagrange3d/lagrange3d3.png}
  & \includegraphics[width=1.25in]{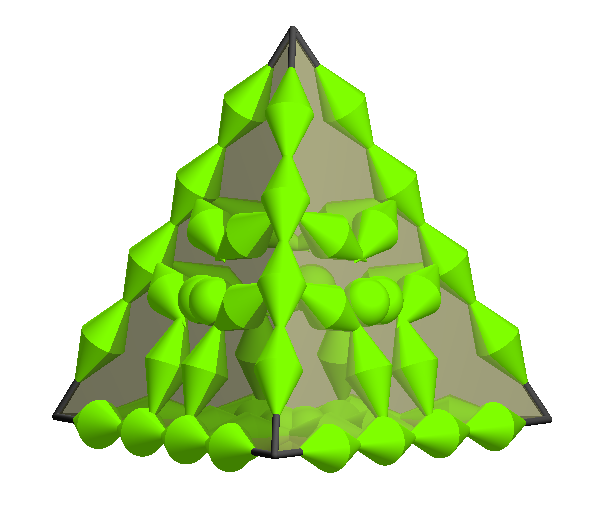}
  & \includegraphics[width=1.25in]{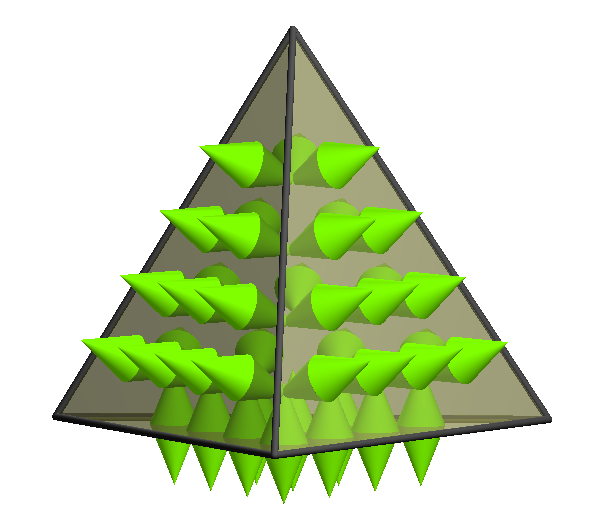}
  & \includegraphics[width=1.25in]{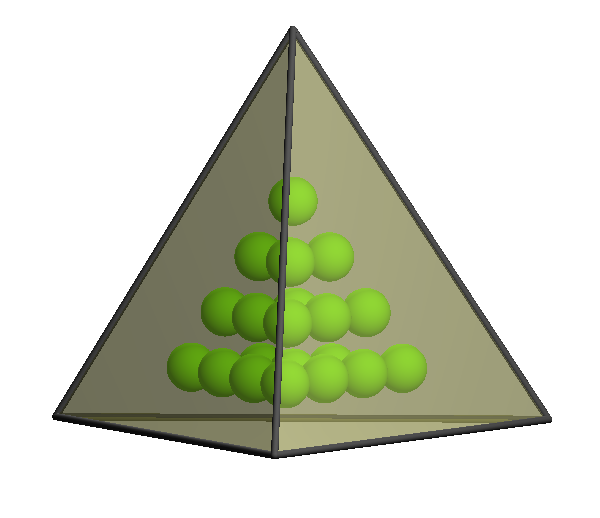}
 \end{tabular}
\end{center}
Legend:
\raise-1pt\hbox{\includegraphics[height=.13in]{figures/symbols/point.png}}~--~scalar (1 DOF);
\raise-4pt\hbox{\includegraphics[height=.187in]{figures/symbols/doublepoint.png}}~--~tangential vector field (2 DOFs);
\raise-4pt\hbox{\includegraphics[height=.187in]{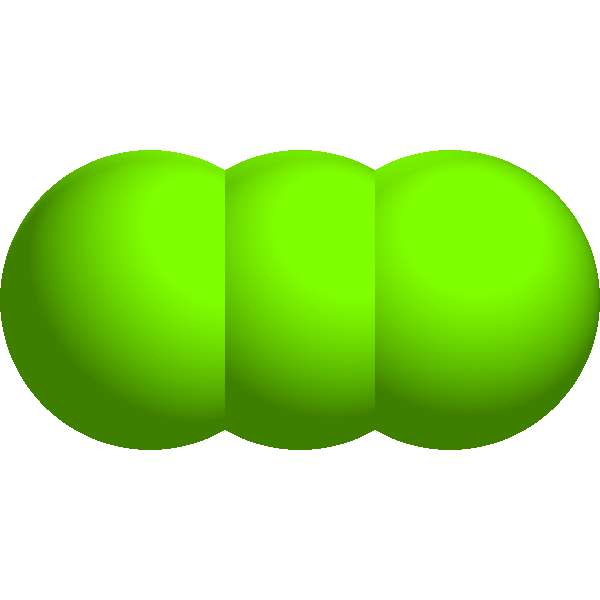}}~--~vector field (2 DOFs);
\raise-5pt\hbox{\includegraphics[height=.245in]{figures/symbols/arrow.png}}~--~tangential component along edge or
normal component on face (1 DOF).
\caption{The $\P_r^-\Lambda^k(\T_h)$ and $\P_r\Lambda^k(\T_h)$ spaces in three dimensions.}\label{fg:simp3d}
\end{figure}

The lowest order spaces $\P_1^-\Lambda^k(\T_h)$ are very geometric, possessing precisely one degree of freedom per face of
dimension $k$, and no others (see the top rows of Figures~\ref{fg:simp2d} and \ref{fg:simp3d}.  In fact these spaces first
appeared in the geometry literature in the work of Whitney in 1957 \cite{Whitney} long before their first appearance
as finite elements.  In the 1970s, they were used by Dodziuk \cite{dodziuk} and Dodziuk and Patodi
\cite{dodziuk-patodi} as a theoretical tool to approximate the
eigenvalues of the Hodge Laplacian on a Riemannian
manifold.  This then played an essential role in M\"uller's proof of the Ray--Singer conjecture \cite{muller}.
The spaces $\P_r\Lambda^k(\T_h)$ also appeared in the geometry literature, introduced by Sullivan \cite{sullivan73,sullivan77}.
In an early, largely overlooked paper bringing finite element analysis techniques to bear on geometry
Baker \cite{baker} named these Sullivan--Whitney forms, and analyzed their convergence for the
eigenvalue problem for the Hodge Laplacian. In 1988 Bossavit made the connection between
Whitney forms and the mixed finite elements in use in electromagnetics \cite{Bossavit}, in part inspired
by the thesis of Kotiuga \cite{kotiuga-thesis}.  The first unified treatment of the $\P_r^-\Lambda^k$ spaces, which
was based on exterior calculus and included a unisolvence proof, was in a seminal paper of Hiptmair \cite{Hiptmair}
in 1999.  In the 2006 paper of Arnold, Falk, and Winther \cite{acta}, in which the term finite element exterior calculus
first appeared, the Koszul complex was first applied to finite elements, simplifying many aspects and resulting
in a simultaneous treatment of both the $\P_r^-\Lambda^k$ and $\P_r\Lambda^k$ spaces.

\section{Families of finite element differential forms on cubical meshes}
We now describe two families of spaces of finite element differential forms,
which we denote $\Q_r^-\Lambda^k(\T_h)$ and $\S_r\Lambda^k(\T_h)$,
defined for cubical meshes $\T_h$, i.e., meshes in which each element is the
Cartesian product of intervals.  In some sense, the $\Q_r^-\Lambda^k$ family
can be seen as an analogue of the $\P_r^-\Lambda^k$ family for simplicial
meshes, and the $\S_r\Lambda^k$ family an analogue of the $\P_r\Lambda^k$
family.  The $\Q_r^-\Lambda^k$ family can be constructed from the
one-dimensional case by a tensor product construction, and is long known.
By contrast, the $\S_r\Lambda^k$ family first appeared in recent work
of Arnold and Awanou \cite{cubicderham}.  Even in two and three dimensions,
the spaces in this family were for the most part not known previously.

\subsection{The $\Q_r^-\Lambda^k$ family}
We describe this family only very briefly.  A more detailed description will
be included in a forthcoming study of the approximation properties of these
spaces under non-affine mappings \cite{arnold-boffi-bonizzoni}. Suppose we are given a subcomplex of the de~Rham complex on
an element $S\subset\R^m$ and a second such subcomplex on an element $T\subset\R^n$:
$$
V^0(S)\xrightarrow{d} V^1(S)\xrightarrow{d}\cdots\xrightarrow{d} V^m(S), \quad
V^0(T)\xrightarrow{d} V^1(T)\xrightarrow{d}\cdots\xrightarrow{d} V^n(T).
$$
We may then construct a subcomplex of the de~Rham complex on $S\x T$ by a tensor
product construction which is known in the theory of differential forms;
see, e.g., \cite[p.~61]{robbin-salamon}.
The canonical projection $\pi_S:S\x T\to S$ determines a pullback of
$i$-forms on $S$ to $i$-forms on $S\x T$, so $\pi^*_S V^i(S)$ is a space of $i$-forms on
$S\x T$ and, similarly, $\pi^*_T V^j(T)$ is a space of $j$-forms on $S\x T$.
Thus we may define a space of $k$-forms on $S\x T$ by 
$$
V^k(S\x T) = \bigoplus_{i+j=k}\pi^*_SV^i(S)\wedge \pi^*_T V^j(T).
$$
We take the space $V^k(S\x T)$ as
the shape functions for $k$-forms on $S\x T$.  The construction
of degrees of freedom for $V^k(S\x T)$ is simple.  If $\eta\in V^i(S)^*$ is a degree
of freedom associated to a face $f$ of $S$, and $\rho\in V^j(T)^*$
is associated to a face $g$ of $T$, we define
$$
\eta\wedge\rho\in [\pi^*_SV^i(S)\wedge \pi^*_T V^j(T)]^*\subset V^k(S\x T)^*,
$$
by
$$
(\eta\wedge\rho)(\pi^*_s u \wedge \pi^*_T v)  = \eta(u)\rho(v),
$$
and associate the degree of freedom $\eta\wedge\rho$ to $f\x g$, which is a face of $S\x T$.

The $\Q_r^-$ family is defined by applying this tensor product repeatedly, starting
with a finite element de~Rham complex on an interval in one dimension.  In one dimension
the $\P_r^-$ and $\P_r$ de~Rham subcomplexes coincide.  On an interval $I$, the
shape functions for $0$-forms are $V^0(I)=\P_r(I)$ with degrees of freedom at each end point,
and moments of degree at most $r-1$ in the interior.  The shape function for $1$-forms
are $V^1(I)=\P_{r-1}(I)$ with all degrees of freedom in the interior.   Repeatedly using the
tensor product construction just outlined, we obtain polynomial spaces and degrees of freedom
on a box $I_1\x\cdots\x I_n\subset\R^n$.  We denote the shape function space so obtained
by $\Q_r^-\Lambda^k(I_1\x\cdots\x I_n)$. In $n=2$ dimensions, for example,
\begin{gather*}
\Q_r^-\Lambda^0(I_1\x I_2) = \Q_r(I_1\x I_2)=\P_r(I_1)\otimes\P_r(I_2),\\
\Q_r^-\Lambda^1(I_1\x I_2) = [\P_{r-1}(I_1)\otimes\P_r(I_2)]\x [\P_r(I_1)\otimes\P_{r-1}(I_2)],\quad
\Q_r^-\Lambda^2(I_1\x I_2)=\Q_{r-1}(I_1\x I_2).
\end{gather*}
Diagrams for these elements in two and three dimensions are shown in Figures~\ref{fg:q}.
The space $\Q_r^-\Lambda^0(\T_h)$ is the standard $\Q_r$ finite element subspace of $H^1(\Omega)$
and the space $\Q_r^-\Lambda^n(\T_h)$ is the discontinuous $\Q_{r-1}$ subspace of $L^2(\Omega)$.
The space $\Q_r^-\Lambda^1(\T_h)$ goes back to Raviart and Thomas \cite{Raviart-Thomas} in two
dimensions, and the $\Q_r^-\Lambda^1(\T_h)$ and $\Q_r^-\Lambda^2(\T_h)$ were given by
N\'ed\'elec in \cite{Nedelec1}.  The spaces with $r$ held fixed combine to create a finite element de~Rham subcomplex,
$$
\Q_r^-\Lambda^0(\T_h)\xrightarrow{d}  \Q_r^-\Lambda^1(\T_h)\xrightarrow{d} \cdots \xrightarrow{d}  \Q_r^-\Lambda^n(\T_h),
$$
and the degrees of freedom determine commuting projections.

\begin{figure}[p]
\begin{center}
\begin{tabular}{cccc}
\hspace{-1.35in}$\Q_r^-\Lambda^k$ (2D): & $k=0$ & $k=1$ & $k=2$
\\[.15in]
\raise.5in\hbox{$r=1$}
  & \includegraphics[width=1.2in]{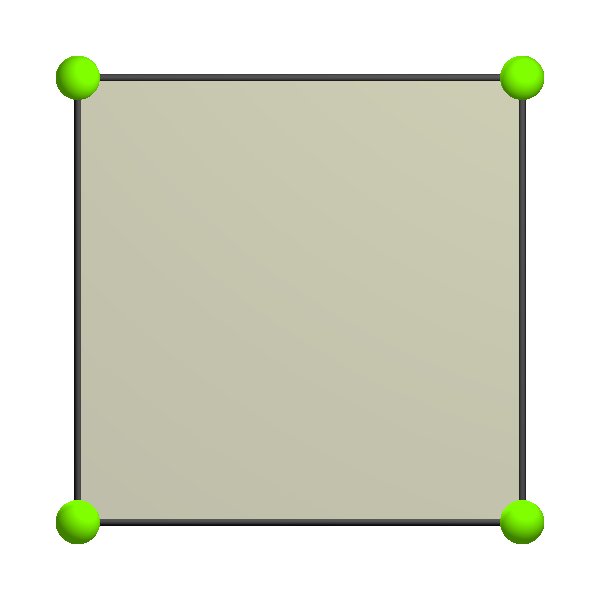}
  & \includegraphics[width=1.2in]{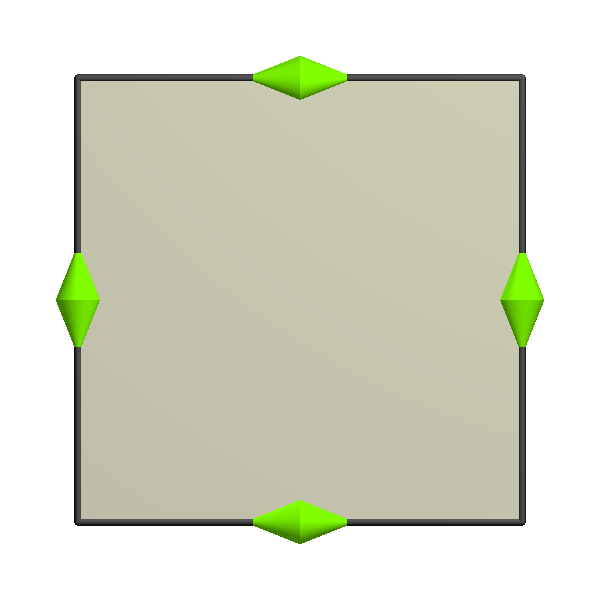}
  & \includegraphics[width=1.2in]{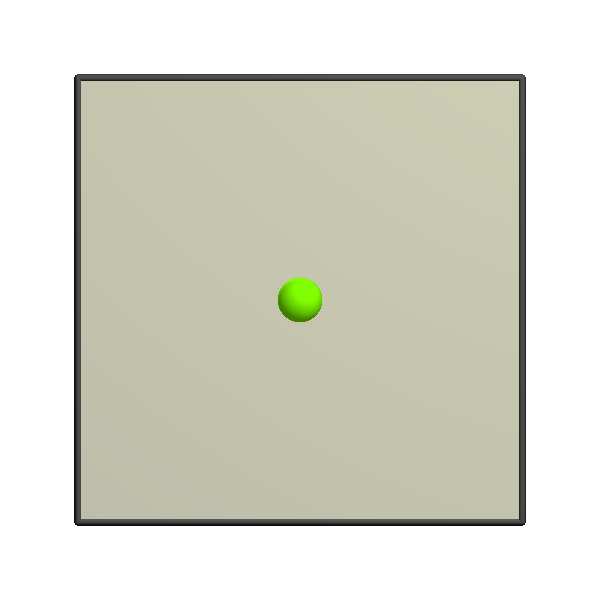}
\\
\raise.5in\hbox{$r=2$}
  & \includegraphics[width=1.2in]{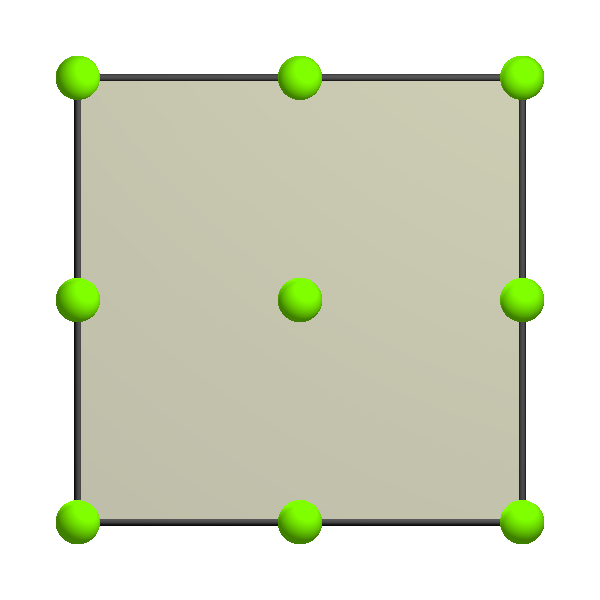}
  & \includegraphics[width=1.2in]{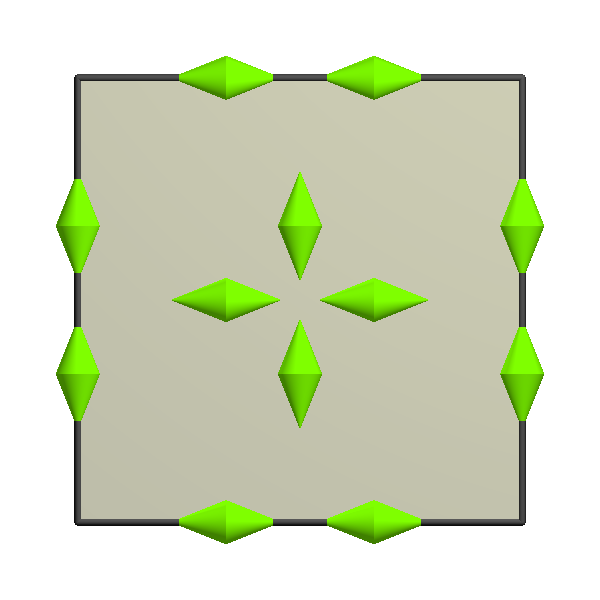}
  & \includegraphics[width=1.2in]{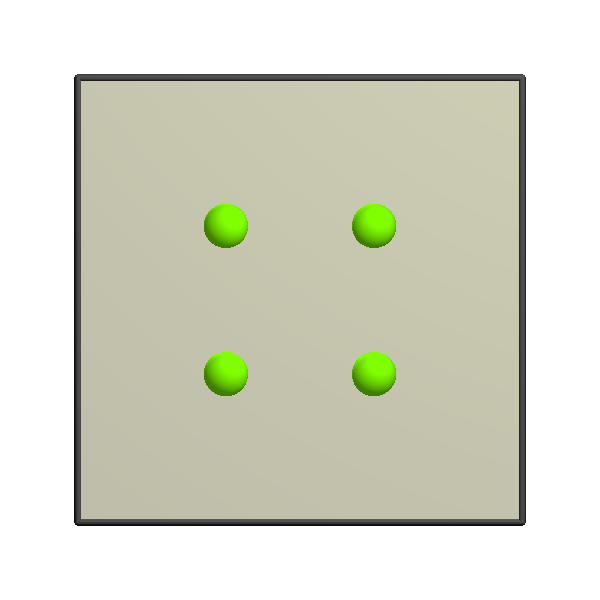}
\\
\raise.5in\hbox{$r=3$}
  & \includegraphics[width=1.2in]{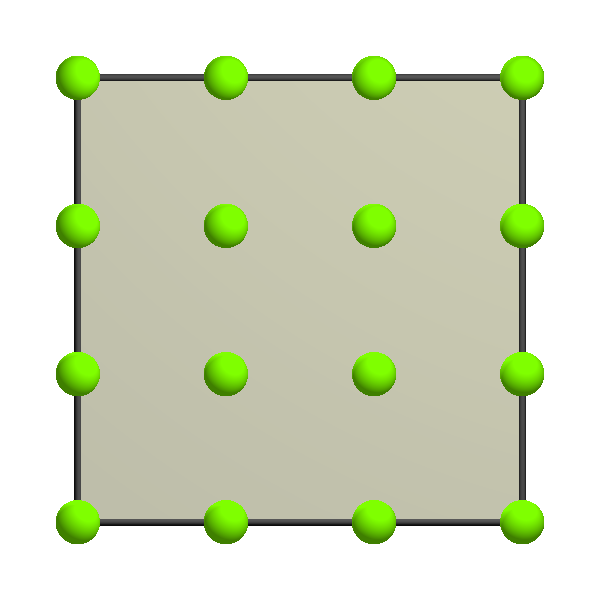}
  & \includegraphics[width=1.2in]{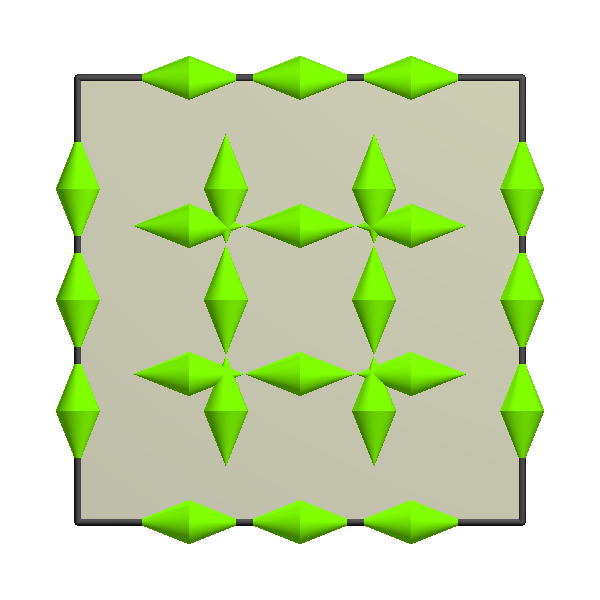}
  & \includegraphics[width=1.2in]{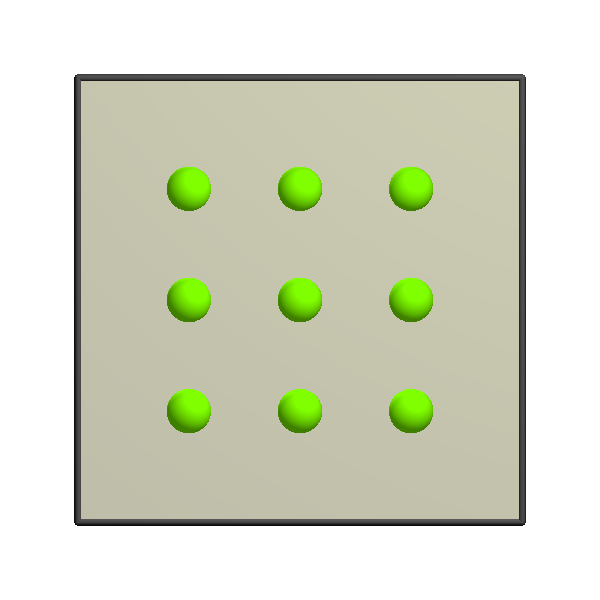}
 \end{tabular}
\begin{tabular}{ccccc}
$\Q_r^-\Lambda^k$ (3D): & $k=0$ & $k=1$ & $k=2$ & $k=3$
\\[.15in]
\raise.5in\hbox{$r=1$}
  & \includegraphics[width=1.2in]{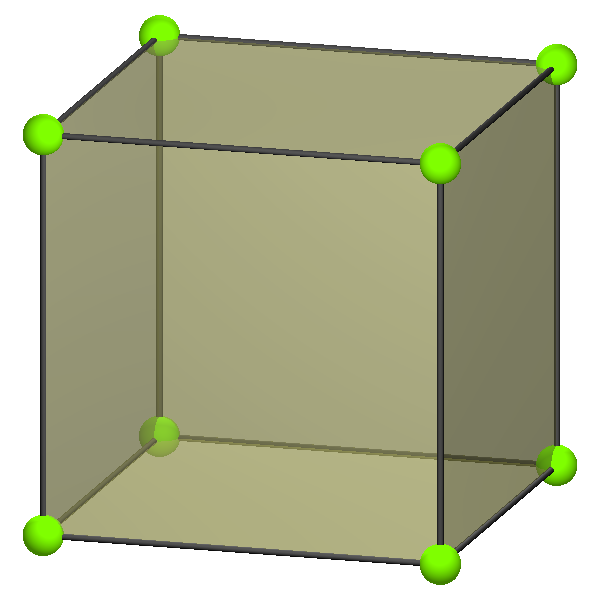}
  & \includegraphics[width=1.2in]{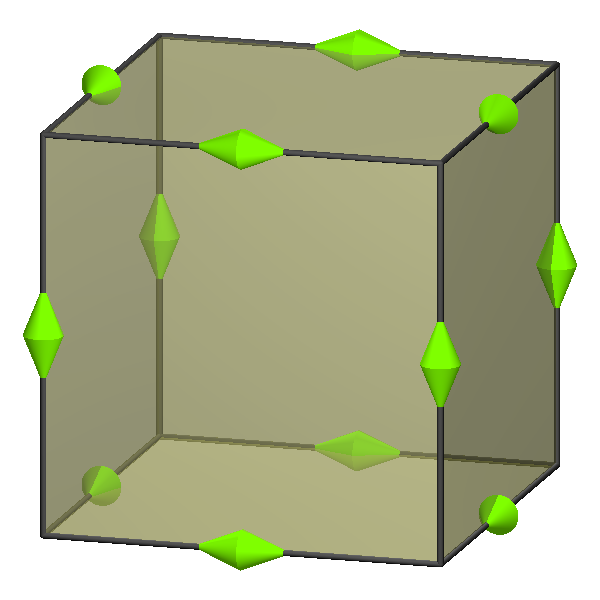}
  & \includegraphics[width=1.2in]{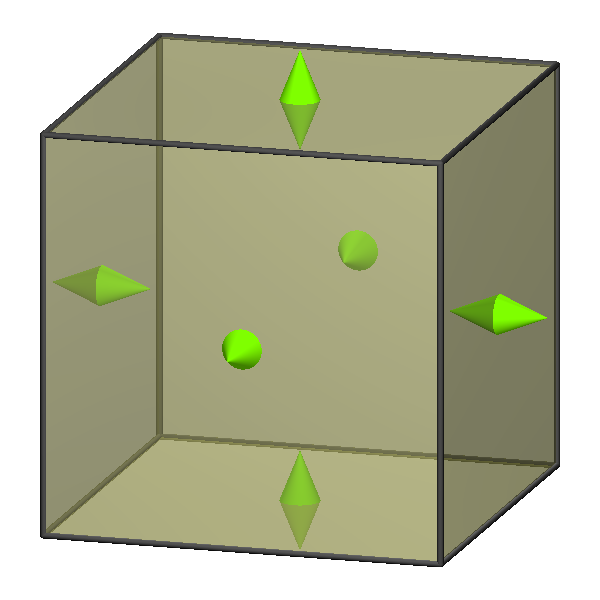}
  & \includegraphics[width=1.2in]{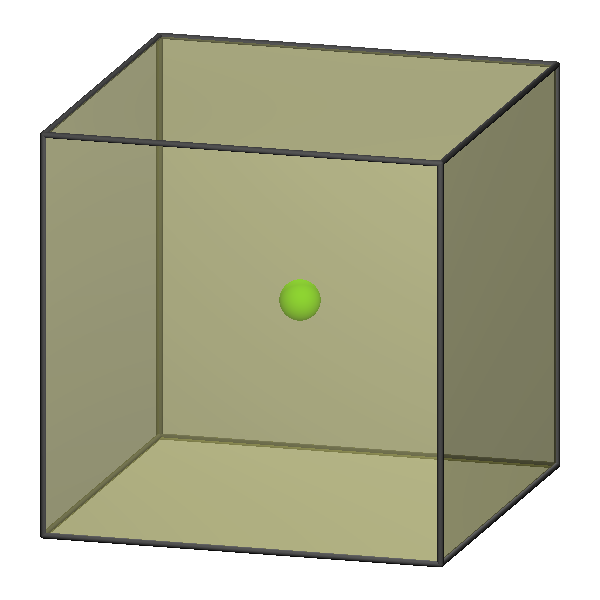}
\\
\raise.5in\hbox{$r=2$}
  & \includegraphics[width=1.2in]{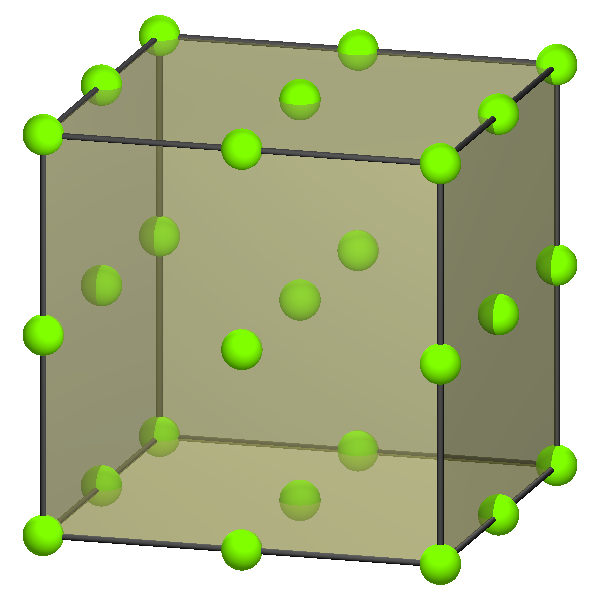}
  & \includegraphics[width=1.2in]{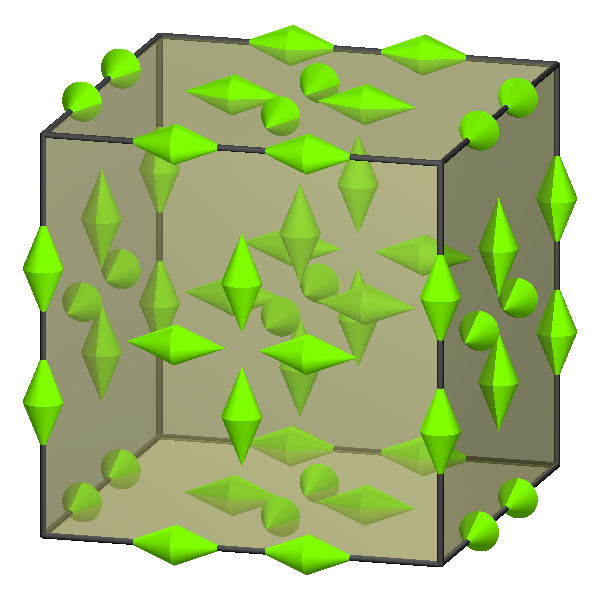}
  & \includegraphics[width=1.2in]{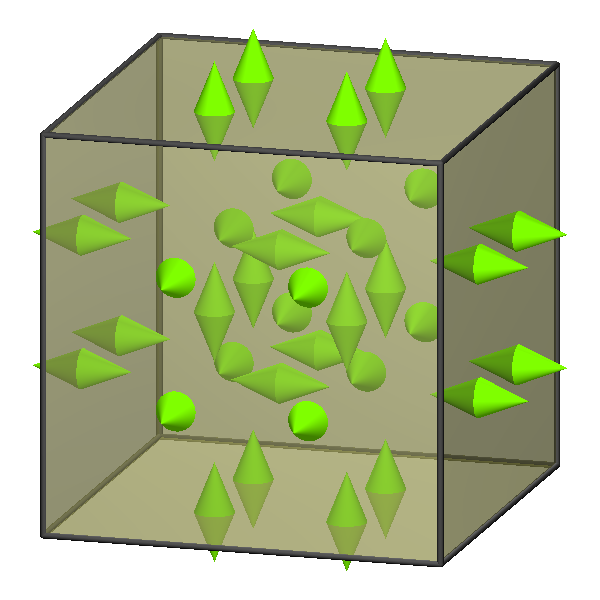}
  & \includegraphics[width=1.2in]{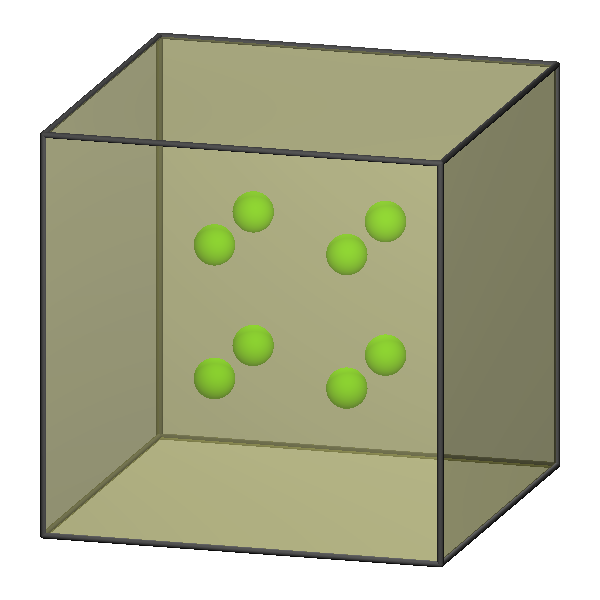}
\\
\raise.5in\hbox{$r=3$}
  & \includegraphics[width=1.2in]{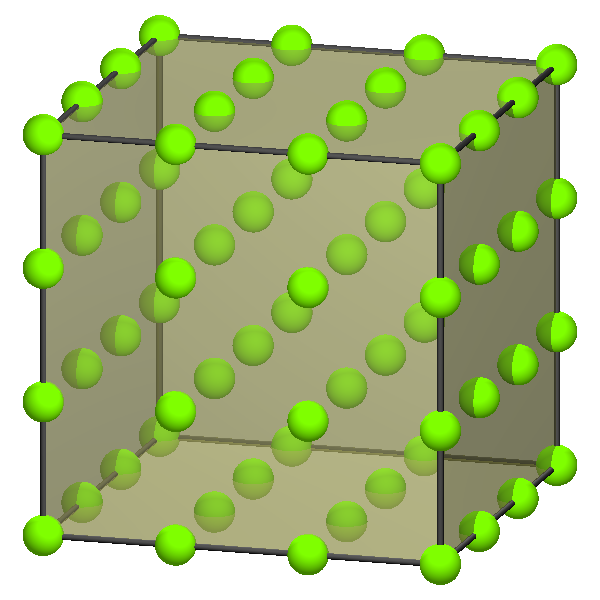}
  & \includegraphics[width=1.2in]{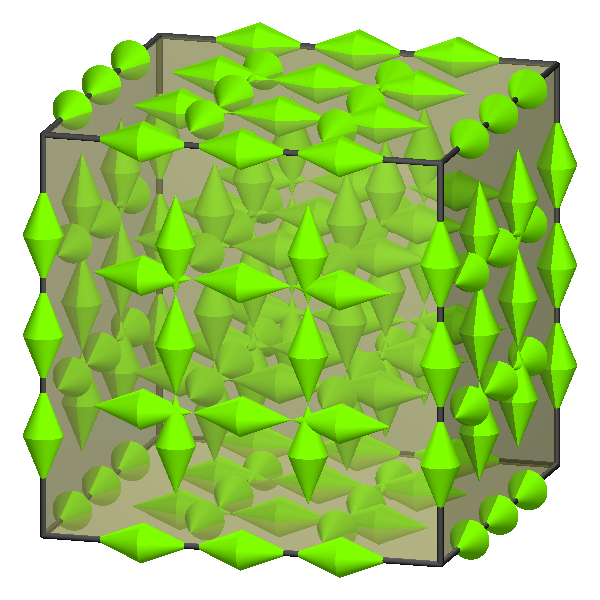}
  & \includegraphics[width=1.2in]{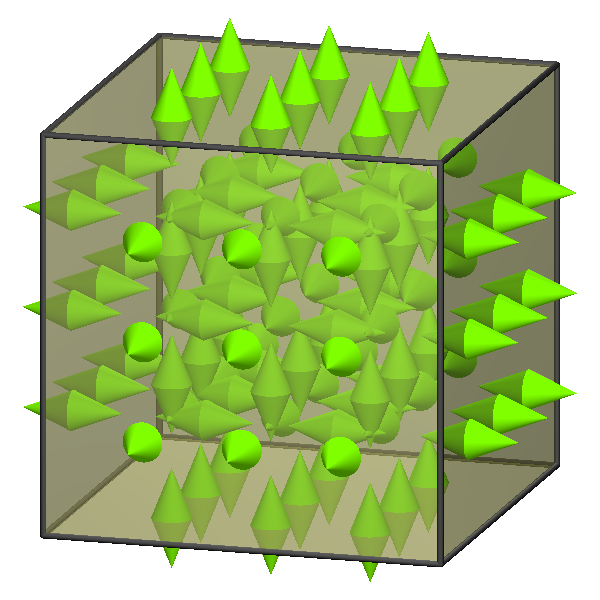}
  & \includegraphics[width=1.2in]{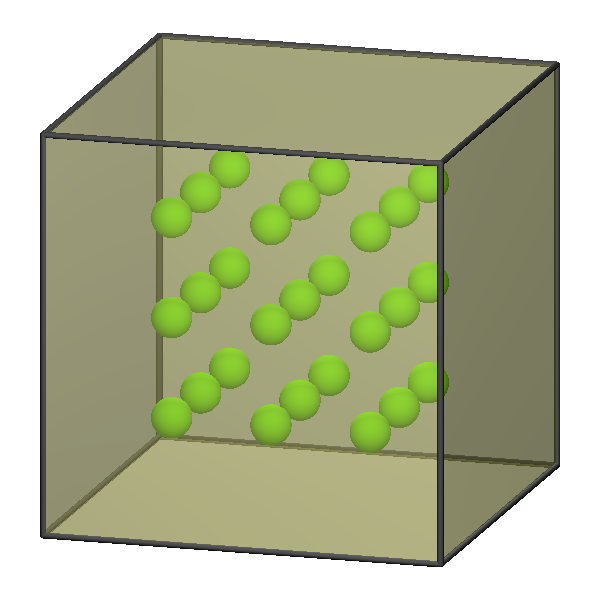}
 \end{tabular}
\end{center}
\caption{The $\Q_r^-\Lambda^k(\T_h)$ spaces in two and three dimensions.}\label{fg:q}
\end{figure}

Recently, Cockburn and Qiu \cite{cockburn-qiu} have published a different family of finite element spaces
in two and three dimensions, that seems to be related to these.
They begin with the complex formed by the full spaces $\Q_r\Lambda^k$, which lie between
$\Q_r^-\Lambda^k$ and $\Q_{r+1}^-\Lambda^k$.  That complex (which was discussed in \cite{Nedelec2})
does \emph{not} admit commuting projections.  Cockburn and Qiu define a small space of bubble functions
that can be added to each of the spaces
so that the resulting spaces remain inside $\Q_{r+1}^-\Lambda^k$ but also form a de~Rham
subcomplex (with constant $r$) which admits commuting projections.

\subsection{A second family of finite element differential forms on cubes}
The $\S_r\Lambda^k$ family presented in this section was derived recently in \cite{cubicderham}.  It seems to
be complementary to the $\Q_r^-\Lambda^k$ family much as the $\P_r\Lambda^k$ family
complements the $\P_r^-\Lambda^k$ family.  To describe the new family
we require some notation.  A $k$-form monomial in $n$ variables is the product of
an ordinary monomial and a simple alternator:
$$
m=(x^1)^{\alpha_1}\cdots(x^n)^{\alpha_n}\, dx^{\sigma_1}\wedge\cdots\wedge dx^{\sigma_k}.
$$
where $\alpha$ is a multi-index and $\sigma\in\Sigma(k,n)$.  We define the degree of
$m$ to be the polynomial degree of its coefficient: $\deg m =\sum_i\alpha_i$.
The \emph{linear degree} of $m$ is more complicated:
$$
\ldeg m = \#\{\,i\,|\,\alpha_i=1,\ \alpha_i\notin\{\sigma_1,\ldots,\sigma_k\}\,\},
$$
that is, the number of variables that enter the coefficient linearly, not counting
the variables that enter the alternator.  For example,
if $m=x^1x^2 (x^3)^5\, dx^1$, then $\deg m=7$, $\ldeg m = 1$.

We now define the space of shape functions we shall use for $k$-forms on 
an $n$-dimensional box, $T$.  Viewing monomial forms as differential
forms on $T$, we define $\H_{r,l}\Lambda^k(T)\subset\H_r\Lambda^k(T)$ to be the span
of all monomial $k$-forms $m$ such that $\deg m=r$ and $\ldeg m\ge l$. Using this
definition and the Koszul differential, we then define
$$
\J_r\Lambda^k(T) = \sum_{l\ge 1}\kappa\H_{r+l-1,l}\Lambda^{k+1}(T)\subset\P_{r+n-k-1}\Lambda^k(T),
$$
Finally, we define the shape functions on $T$ by
$$
\S_r\Lambda^k(T) = \P_r\Lambda^k(T) + \J_r\Lambda^k(T) + d\J_{r+1}\Lambda^{k-1}(T),
$$
defined for all $r\ge 1$, $0\le k\le n$.

As the definition of the shape functions takes a while to absorb, we describe the spaces
in more elementary terms in the case of three dimensions.
\begin{itemize}
 \item
The space $\S_r\Lambda^0$, the polynomial shape functions
for the $H^1$ space, consists of all polynomials $u$ with \emph{superlinear degree} $\sdeg u\le r$.
The superlinear degree of a monomial is its degree ignoring any variable that enters to the first
power, and the superlinear degree of a polynomial is the maximum over its monomials.
The criterion $\sdeg u\le r$ was introduced in \cite{serendipity} to generalize the
serendipity elements from $2$ to $n$-dimensions.
\item The space $\S_r\Lambda^1$, the shape functions for the $H(\curl)$ space, consists
of vector fields of the form
$$
(v^1,v^2,v^3)  + (x^2x^3(w^2-w^3),x^3x^1(w^3-w^1),x^1x^2(w^1-w^2)) +\grad u,
$$
with polynomials $v^i$, $w^i$, and $u$ for which $\deg v^i\le r$, $\deg w^i\le r-1$, $\sdeg u\le r+1$, and
$w^i$ is independent of the variable $x^i$.
\item The $H(\div)$ space uses shape functions $\S_r\Lambda^2$, which are of the form
$$
(v^1,v^2,v^3)  + \curl (x^2x^3(w^2-w^3),x^3x^1(w^3-w^1),x^1x^2(w^1-w^2)),
$$
with $\deg v^i\le r$, $\deg w^i\le r$, and $w^i$ independent of the variable $x^i$.
\item Finally the $L^2$ space $\S_r\Lambda^3$ simply coincides with $\P_r$.

\end{itemize}

In \cite{cubicderham} we establish the following properties of these spaces (in any dimension):
\begin{itemize}
 \item degree property:\quad $\P_r\Lambda^k(I\,^n)\subset \S_r\Lambda^k(I\,^n)\subset  \P_{r+n-k}\Lambda^k(I^n)$,

\item inclusion property:\quad $\S_r\Lambda^k(I\,^n) \subset \S_{r+1}\Lambda^k(I\,^n)$,

\item trace property:\quad For each face $f$ of $I\,^n$,  $\tr_f \S_r\Lambda^k(I\,^n)\subset\S_r\Lambda^k(f)$,

\item subcomplex property:\quad $d\S_r\Lambda^k(I\,^n)\subset \S_{r-1}\Lambda^{k+1}(I\,^n)$.
\end{itemize}

The degrees of freedom for $\S_r\Lambda^k(T)$ are quite simple:
\begin{equation}\label{sdofs}
 u\in\S_r\Lambda^k(T)\mapsto \int_f (\tr_f u)\wedge q, \quad q\in\P_{r-2(d-k)}\Lambda^{d-k}(f), \ f\in\Delta_d(T), \ d\ge k.
\end{equation}
These are illustrated in Figure~\ref{fg:s}.
Notice that weighting function $q$ is sought in a $\P_s$ space, not a $\Q_s$ space.  Moreover, as the
face dimension $d$ increases by $1$, the degree $s$ of the space used for $q$ decreases by $2$.
A major result of \cite{cubicderham} is a proof that the degrees of freedom are unisolvent.
Further, we show there that the resulting finite element spaces combine into de~Rham subcomplexes
with commuting projections:
$$
\S_r\Lambda^0(\T_h)\xrightarrow{d}S_{r-1}\Lambda^1(\T_h)\xrightarrow{d}\cdots \xrightarrow{d}S_{r-n}\Lambda^n(\T_h),
$$
in which the degrees $r$ decrease, as for the $\P_r\Lambda^K(\T_h)$ spaces on simplices.
\begin{figure}[p]
\begin{center}
\begin{tabular}{cccc}
\hspace{-1.35in}$\S_r\Lambda^k$ (2D): & $k=0$ & $k=1$ & $k=2$
\\[.15in]
\raise.5in\hbox{$r=1$}
  & \includegraphics[width=1.2in]{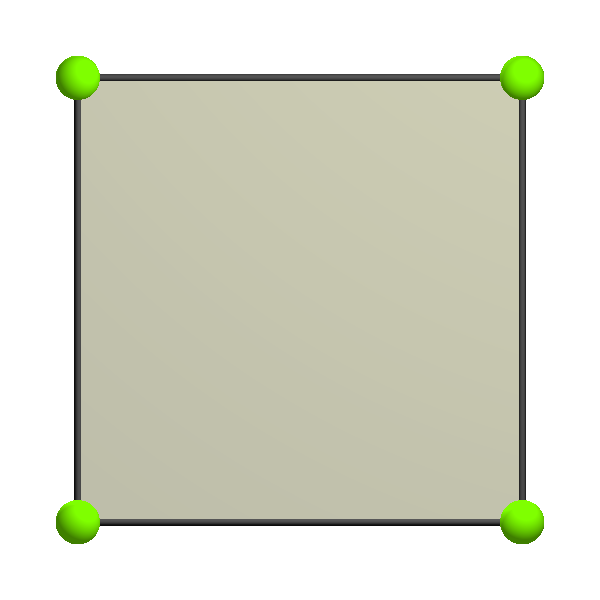}
  & \includegraphics[width=1.2in]{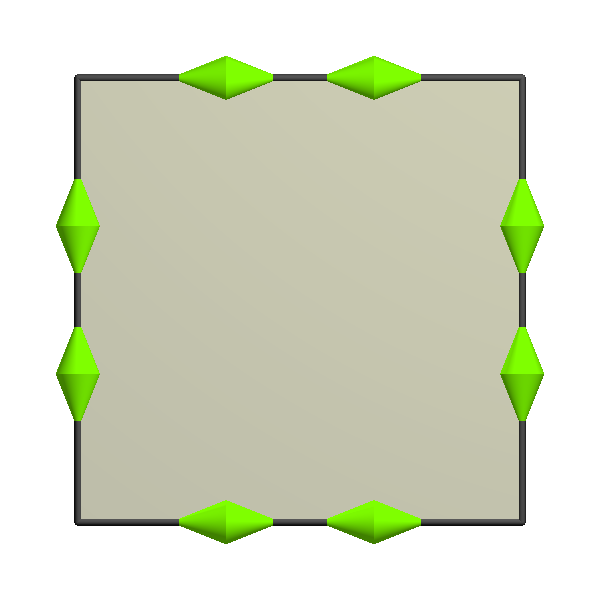}
  & \includegraphics[width=1.2in]{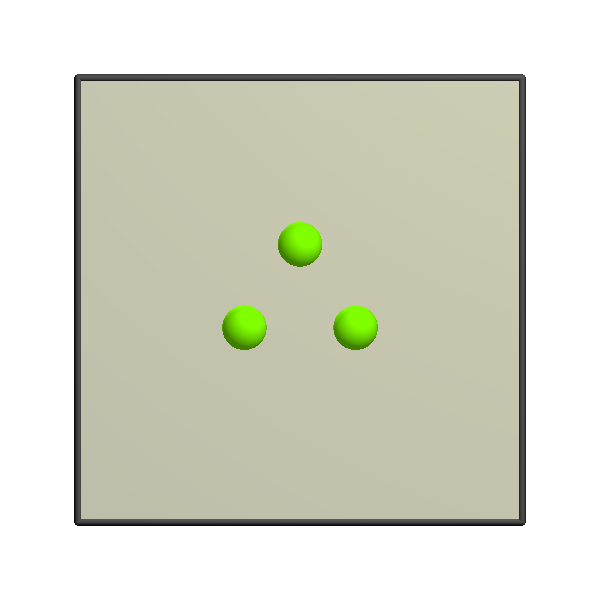}
\\
\raise.5in\hbox{$r=2$}
  & \includegraphics[width=1.2in]{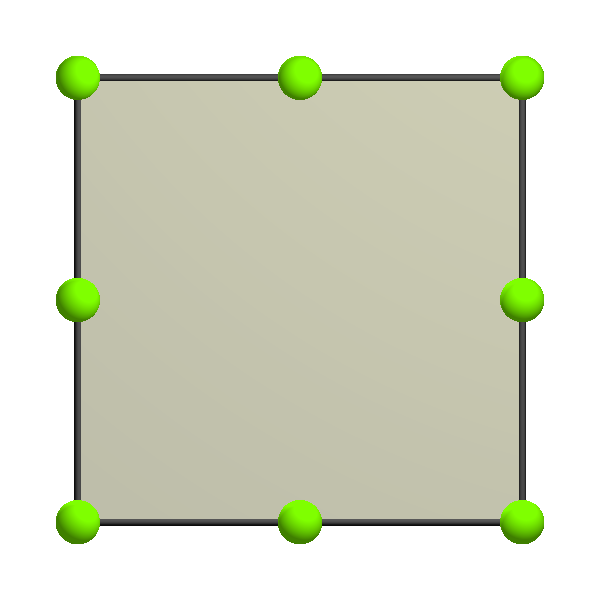}
  & \includegraphics[width=1.2in]{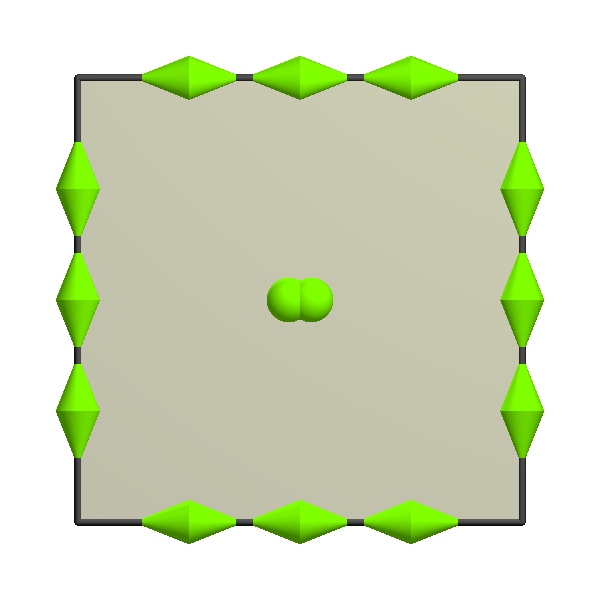}
  & \includegraphics[width=1.2in]{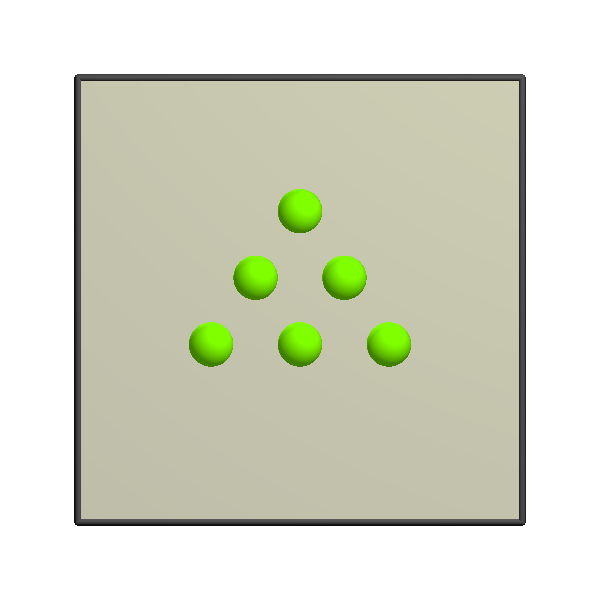}
\\
\raise.5in\hbox{$r=3$}
  & \includegraphics[width=1.2in]{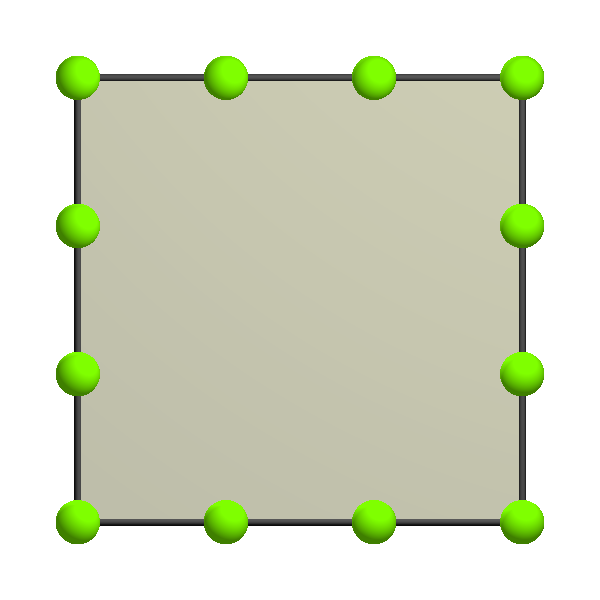}
  & \includegraphics[width=1.2in]{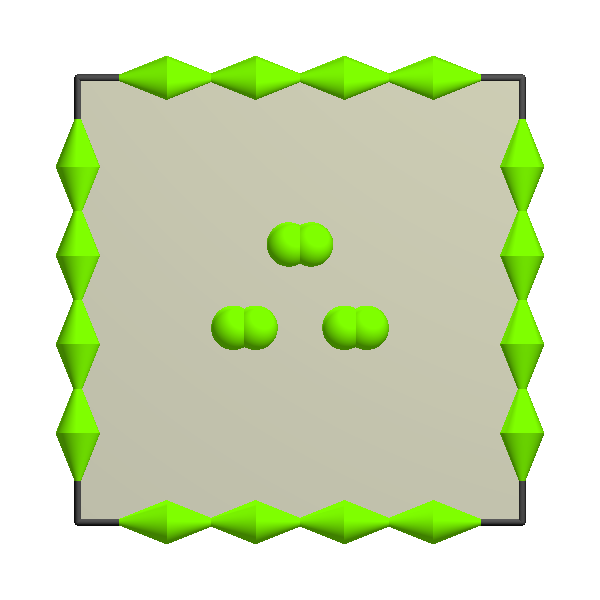}
  & \includegraphics[width=1.2in]{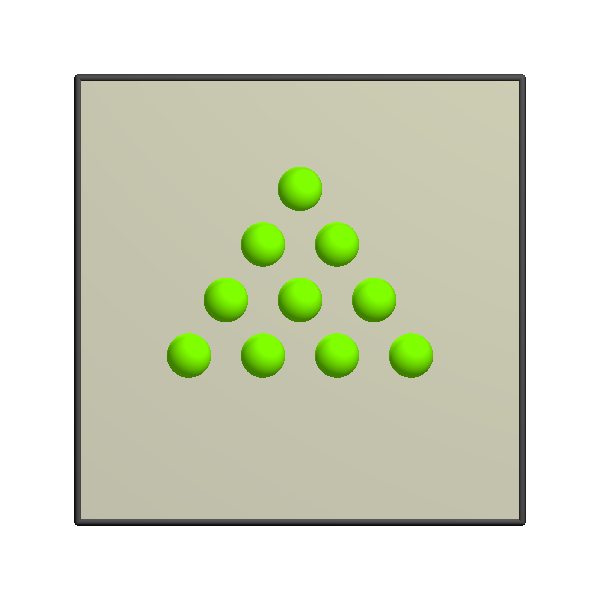}
 \end{tabular}
\begin{tabular}{ccccc}
$\S_r\Lambda^k$ (3D): & $k=0$ & $k=1$ & $k=2$ & $k=3$
\\[.15in]
\raise.5in\hbox{$r=1$}
  & \includegraphics[width=1.2in]{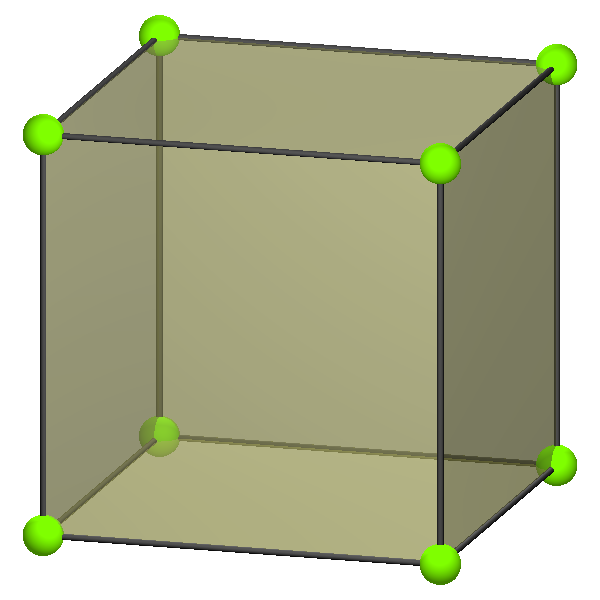}
  & \includegraphics[width=1.2in]{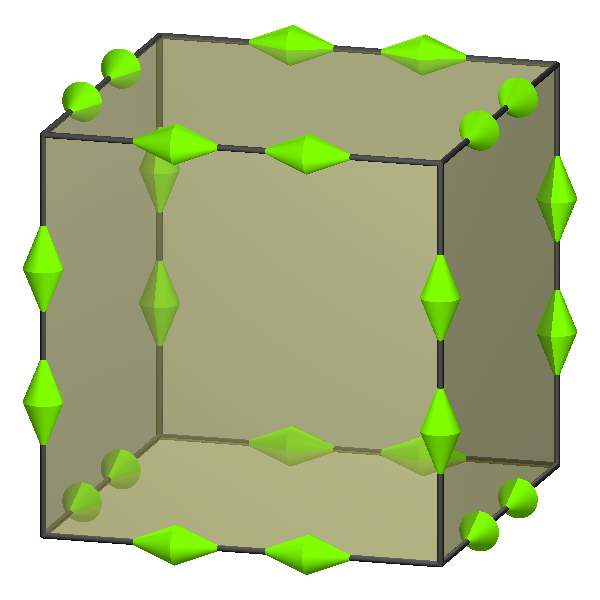}
  & \includegraphics[width=1.2in]{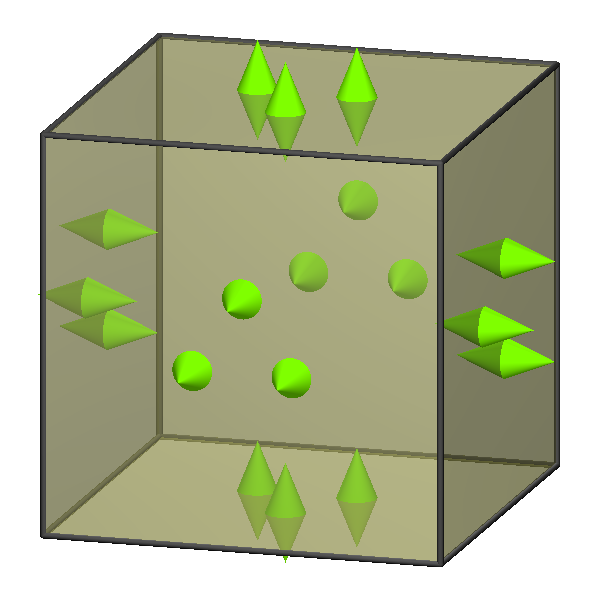}
  & \includegraphics[width=1.2in]{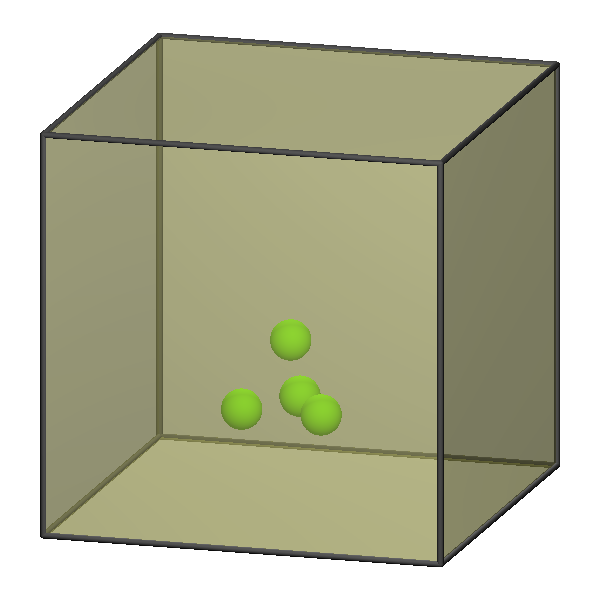}
\\
\raise.5in\hbox{$r=2$}
  & \includegraphics[width=1.2in]{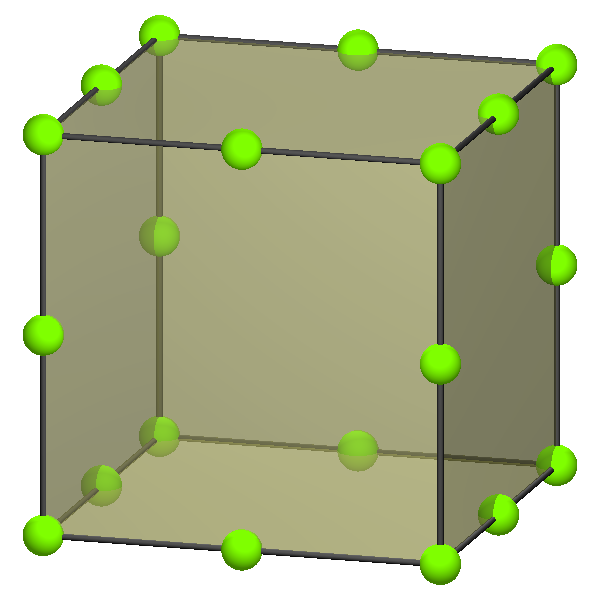}
  & \includegraphics[width=1.2in]{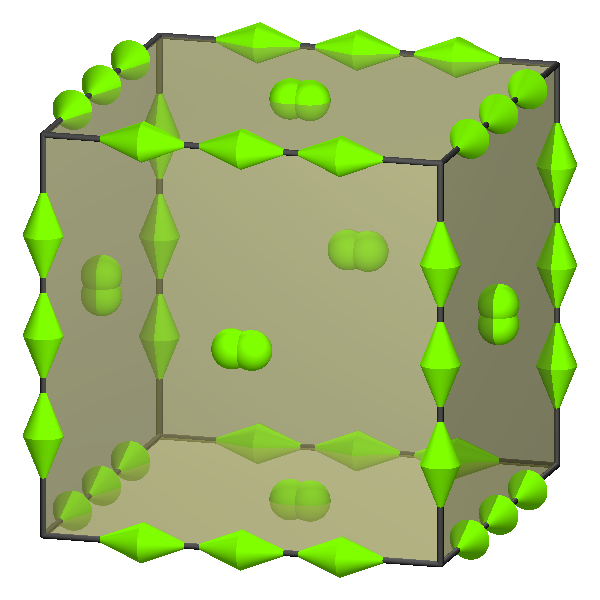}
  & \includegraphics[width=1.2in]{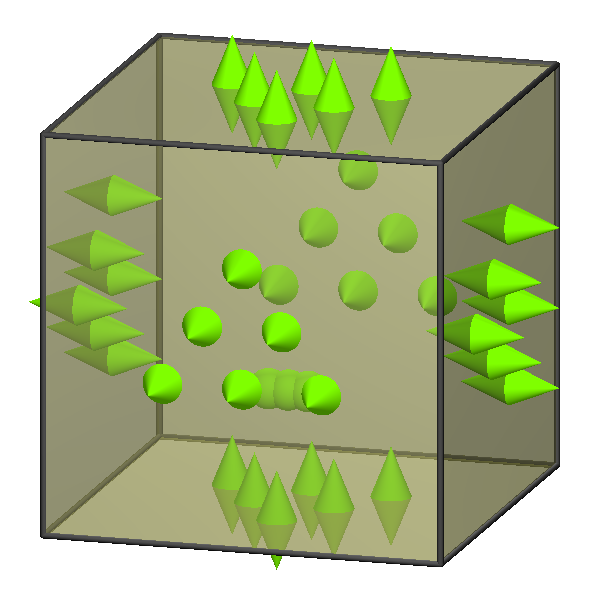}
  & \includegraphics[width=1.2in]{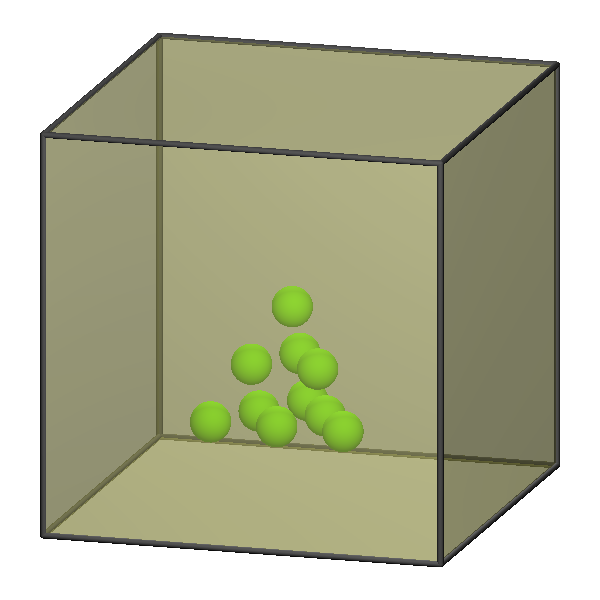}
\\
\raise.5in\hbox{$r=3$}
  & \includegraphics[width=1.2in]{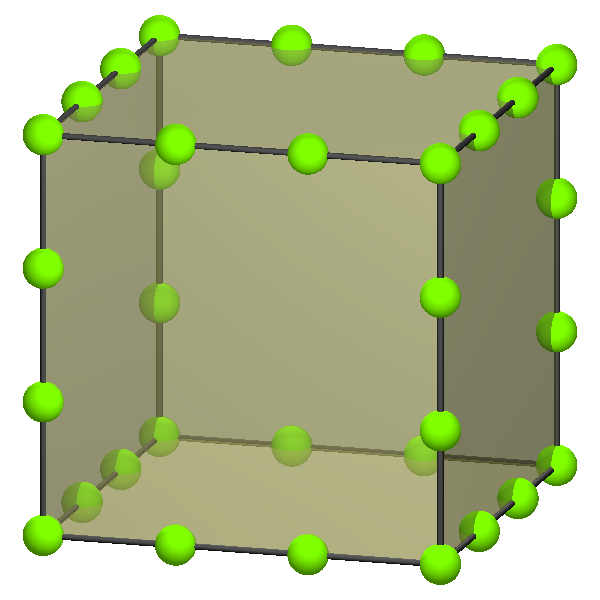}
  & \includegraphics[width=1.2in]{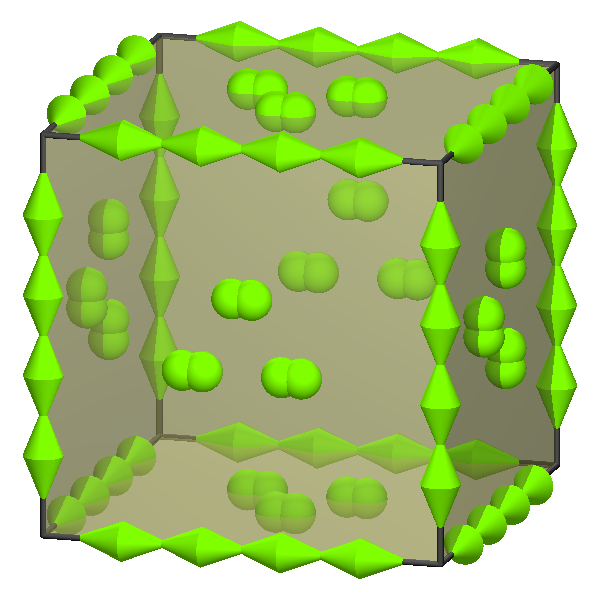}
  & \includegraphics[width=1.2in]{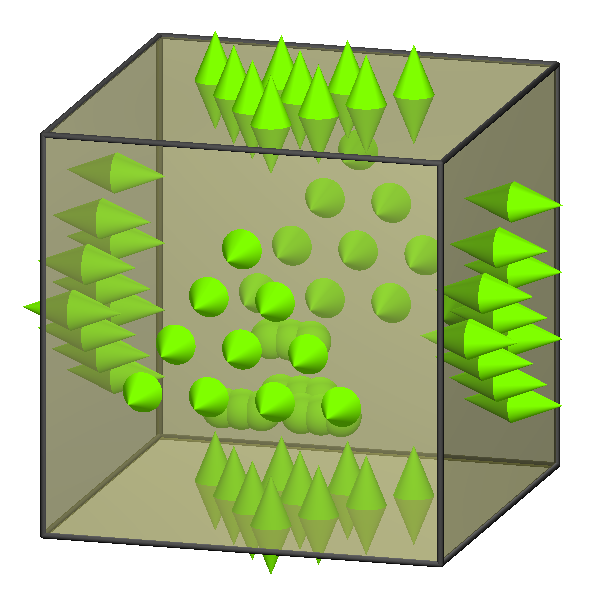}
  & \includegraphics[width=1.2in]{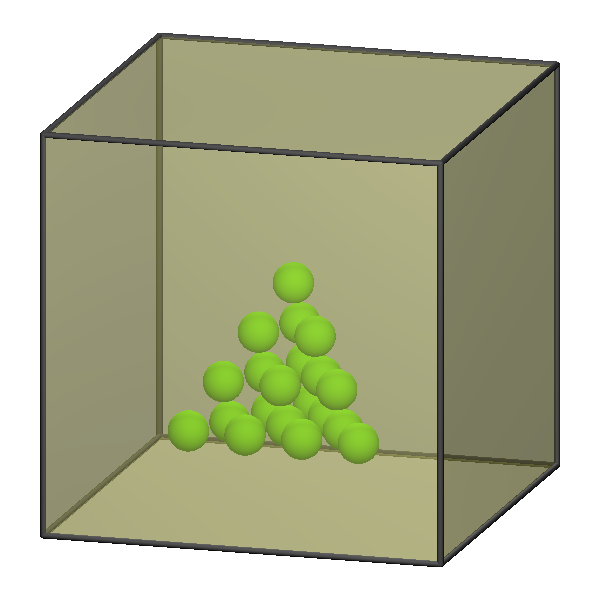}
 \end{tabular}
\end{center}
\caption{The $\S_r\Lambda^k(\T_h)$ spaces in two and three dimensions.}\label{fg:s}
\end{figure}

For $n$-forms, the space $\S_r\Lambda^n(\T_h)$ is simply the discontinuous $\P_r$ space
(but defined on boxes, rather than simplices).
In $2$-dimensions, the $0$-form space $\S_r\Lambda^0(\T_h)$ is the well-known serendipity space,
and the $1$-form space is the rectangular BDM space defined in \cite{Brezzi-Douglas-Marini}.
Hence these spaces were all known in $2$ dimensions.  However, in $3$ and more dimensions
they were not.
The $0$-form space is the appropriate generalization of the serendipity space
to higher dimensions, a space first defined in 2011 \cite{serendipity}.  The space
$\S_r\Lambda^2$ in 3-D is, we believe, the correct analogue of the BDM elements to
cubical meshes.  It has the same degrees of freedom as the space in \cite{Brezzi-Douglas-Duran-Fortin}
but the shape functions have better symmetry properties.  For $1$-forms in 3-D,
$\S_r\Lambda^1$ is a finite element discretization of $H(\curl)$.  To the best of our knowledge,
neither the degrees of freedom nor the shape functions for this space
had been proposed previously.  Finally, we note that the dimension of $\S_r^-\Lambda^k(T)$
tends to be much smaller than that of $\Q_r^-\Lambda^k(T)$, especially for $r$ large, as can
be observed in Table~\ref{tb:dims}.

\begin{table}[ht]
\centerline{%
\begin{tabular}{r||rrrrrrr}
 \multicolumn{7}{c}{$r$}\\
$k$ &  \quad 1 & \quad 2 & \quad 3 & \quad 4 & \quad 5 & \quad 6 \\
\hline
\multicolumn{7}{r}{\rlap{$n=1$}}\\
\hline
 0 &  2 & 3 & 4 & 5 & 6 & 7 \\
 1 &  1 & 2 & 3 & 4 & 5 & 6 \\
\hline
\multicolumn{7}{r}{\rlap{$n=2$}}\\
\hline
 0 &  4 & 9 & 16 & 25 & 36 & 49  \\
 1 &  4 & 12 & 24 & 40 & 60 & 84 \\
 2 &  1 & 4 & 9 & 16 & 25 & 36 \\
\hline
\multicolumn{7}{r}{\rlap{$n=3$}}\\
\hline
 0 &  8 & 27 & 64 & 125 & 216 & 343  \\
 1 &  12 & 54 & 144 & 300 & 540 & 882 \\
 2 &  6 & 36 & 108 & 240 & 450 & 756  \\
 3 &  1 & 8 & 27 & 64 & 125 & 216     \\
\hline
\multicolumn{7}{r}{\rlap{$n=4$}}\\
\hline
 0 &  16 & 81 & 256 & 625 & 1296 & 2401  \\
 1 &  32 & 216 & 768 & 2000 & 4320 & 8232 \\
 2 &  24 & 216 & 864 & 2400 & 5400 & 10584  \\
 3 &  8 & 96 & 432 & 1280 & 3000 & 6048  \\
 4 &  1 & 16 & 81 & 256 & 625 & 1296 
\\
\multicolumn{2}{r}{}
\end{tabular}\qquad
\begin{tabular}{rrrrrr}
 \multicolumn{6}{c}{$r$}\\
 1 & \quad 2 & \quad 3 & \quad 4 & \quad 5 & \quad 6 \\
\hline
\\
\hline
2 & 3 & 4 & 5 & 6 & 7 \\
 2 & 3 & 4 & 5 & 6 & 7 \\
\hline
\\
\hline
4 & 8 & 12 & 17 & 23 & 30  \\
8 & 14 & 22 & 32 & 44 & 58 \\
3 & 6 & 10 & 15 & 21 & 28 \\
\hline
\\
\hline
8 & 20  & 32 & 50 & 74 & 105 \\
24 & 48 & 84 & 135 & 204 & 294 \\
18 & 39 & 72 & 120 & 186 & 273 \\
4 & 10 & 20 & 35 & 56 & 84 \\
\hline
\\
\hline
16 & 48 & 80 & 136 & 216 & 328 \\
64 & 144 & 272 & 472 & 768 & 1188 \\
72 & 168 & 336 & 606 & 1014 & 1602 \\
32 & 84 & 180 & 340 & 588 & 952 \\
5 & 15 & 35 & 70 & 126 & 210
\\
\\
\end{tabular}}
\caption{Dimension of $\Q_r^-\Lambda^k(I^n)$ and $\S_r\Lambda^k(I^n)$.}\label{tb:dims}
\end{table}

\bibliographystyle{amsplainurl}
\bibliography{fe-families}

\end{document}